\documentclass[UKenglish, letterpaper]{amsart} 

\usepackage{amscd}
\usepackage{amsmath}
\usepackage{amssymb}
\usepackage{babel}
\usepackage{cancel}
\usepackage[T1]{fontenc}
\usepackage[utf8]{inputenc}
\usepackage{hyperref}
\usepackage[hyperpageref]{backref}
\usepackage{varioref}
\usepackage[arrow,curve,matrix]{xy}
\usepackage[retainorgcmds]{IEEEtrantools}

\usepackage{mathpazo}
\usepackage{mathrsfs}

\usepackage{times}
\usepackage{graphicx}
\usepackage{color}

\def\clap#1{\hbox to 0pt{\hss#1\hss}}

\DeclareFontFamily{OMS}{rsfs}{\skewchar\font'60}
\DeclareFontShape{OMS}{rsfs}{m}{n}{<-5>rsfs5 <5-7>rsfs7 <7->rsfs10 }{}
\DeclareSymbolFont{rsfs}{OMS}{rsfs}{m}{n}
\DeclareSymbolFontAlphabet{\scr}{rsfs}

\newcommand{\Q}{\mathbb{Q}}

\newcommand{\sE}{\scr{E}}
\newcommand{\sF}{\scr{F}}
\newcommand{\sG}{\scr{G}}

\newcommand{\sO}{\scr{O}}

\newcommand{\sQ}{\scr{Q}}

\newcommand{\sT}{\scr{T}}

\newcommand{\bC}{\mathbb{C}}

\newcommand{\bP}{\mathbb{P}}
\newcommand{\bQ}{\mathbb{Q}}
\newcommand{\bR}{\mathbb{R}}

\DeclareMathOperator{\an}{an}

\DeclareMathOperator{\codim}{codim}

\DeclareMathOperator{\Gal}{Gal}

\DeclareMathOperator{\rank}{rank}

\DeclareMathOperator{\reg}{reg}

\DeclareMathOperator{\id}{{id}}
\DeclareMathOperator{\chern}{{c}}
\DeclareMathOperator{\ch}{{ch}}

\newcommand{\into}{\hookrightarrow}

\newcommand{\wtilde}{\widetilde}

\newcounter{thisthm}

\newcommand{\iref}[1]{(\thesection.\the\value{thisthm}.\the\value{#1})}

\theoremstyle{plain}    
\newtheorem{thm}{Theorem}[section]

\numberwithin{equation}{thm}
\numberwithin{figure}{section}
\theoremstyle{plain}    
\newtheorem{cor}[thm]{Corollary}
\newtheorem{lem}[thm]{Lemma}

\theoremstyle{plain}    
\newtheorem{prop}[thm]{Proposition}
\newtheorem{proclaim-special}[thm]{\specialthmname}

\theoremstyle{remark}
\newtheorem{setup}[thm]{Setup} 

\newtheorem{defn}[thm]{Definition}
\newtheorem{rem}[thm]{Remark}
\newtheorem{explanation}[thm]{Explanation}

\newtheorem{subclaim}[equation]{Claim} 

\newtheorem*{claim*}{Claim}

\newtheoremstyle{bozont-remark}{3pt}{3pt}%
     {}
     {}
     {\it}
     {.}
     {.5em}
     {\thmname{#1}\thmnumber{ #2}: \thmnote{\sc #3}}
\theoremstyle{bozont-remark}

\def\factor#1.#2.{\left. \raise 2pt\hbox{$#1$} \right/\hskip -2pt\raise
  -2pt\hbox{$#2$}}

\newlength{\swidth}
\newenvironment{enumerate-p}{
  \begin{enumerate}}
  {\setcounter{equation}{\value{enumi}}\end{enumerate}}

\definecolor{tomato}{RGB}{180,62,39}
\definecolor{forrest}{RGB}{81,133,49}
\definecolor{lighttomato}{RGB}{253,65,65}
\definecolor{lightforrest}{RGB}{145,237,87}
\definecolor{mygreen}{RGB}{40,104,69}
\definecolor{mygreen2}{RGB}{3,149,39}
\definecolor{darkolivegreen}{RGB}{102,118,75}
\definecolor{cranegreen}{RGB}{102,118,75}
\definecolor{mydarkblue}{RGB}{10,92,153}
\definecolor{myblue}{RGB}{57,222,186}
\definecolor{pinkish}{RGB}{213,83,222}
\definecolor{colD}{RGB}{213,83,222}
\definecolor{defb}{RGB}{213,83,222}
\definecolor{goldenrod}{RGB}{225,115,69}
\definecolor{mauve}{RGB}{224, 176, 255}
\definecolor{fuchsia}{RGB}{255, 0, 255}
\definecolor{lavender}{RGB}{230, 230, 250}
\definecolor{gold}{RGB}{255, 215, 0}
\definecolor{orange}{RGB}{255, 127, 0}
\definecolor{maroon}{RGB}{123, 17, 19}
\definecolor{brightmaroon}{RGB}{195, 33, 72}
\definecolor{richmaroon}{RGB}{176, 48, 96}
\definecolor{green}{RGB}{3,149,39}

\title[A characterization of quotients of Abelian varieties]{A characterization of finite quotients of Abelian varieties}
\author{Steven Lu and Behrouz Taji}

\address{Steven Lu, The Department of Mathematics, University of Quebec at Montreal, succursale centre-ville, Montreal, QC, H3C 3P8, Canada }
\email{\href{mailto:lu.steven@uqam.ca}{lu.steven@uqam.ca}}

\address{Behrouz Taji, Albert-Ludwigs-Universität Freiburg,  Mathematisches Institut,  Eckerstra{\ss}e 1, 79104 Freiburg, Germany}

\email{\href{mailto:behrouz.taji@math.uni-freiburg.de }{behrouz.taji@math.uni-freiburg.de }}

\thanks{Steven Lu was supported by NSERC and TIMS during the write-up of this paper.
  Behrouz Taji was partially supported by the DFG-Graduiertenkolleg GK1821
  ``Cohomological Methods in Geometry'' at Freiburg.}



\hypersetup{
    pdfstartview={Fit},
  pdfpagelayout={TwoColumnRight},
  pdfpagemode={UseOutlines},
  bookmarks,
  colorlinks}

\begin{document}

\begin{abstract} We provide a characterization of quotients of Abelian varieties by finite groups actions that are free in codimension-one via vanishing conditions on the orbifold Chern classes. The characterization is given among a class of varieties with singularities that are more general than quotient singularities, namely among the class of klt varieties.
Furthermore, for a semistable (respectively stable) reflexive ${\mathcal O}_{X}$-module $\sE$ with zero first and second orbifold Chern classes over such a variety $X$, we show that $\sE|_{X_{\reg}^{\an}}$ is locally-free and flat, given by a linear (irreducible unitary) representation of $\pi_1(X_{\reg})$, and that it extends over a finite Galois cover $\wtilde X$ of $X$ \'etale over $X_{\reg}$ to a locally-free and flat sheaf given by an equivariant linear (irreducible unitary) representation of $\pi_1(\wtilde X)$. These are generalizations to the singular setting that is more general than any orbifold strengthenings of the classical correspondences of Donaldson-Uhlenbeck-Yau-Simpson.
%

\end{abstract}

\maketitle

\vspace{-3mm}

\section{Introduction}

We provide sufficient conditions for a \emph{(semi)stable} reflexive sheaf over a normal projective varietiy $X$ to be locally-free and flat on the smooth locus of $X$ and also across the singular locus up to a finite cover of $X$. 
The latter result depends for this paper crucially on a recent result of \cite{GKP14a} on the existence of a suitable finite cover (in the case $X$ is klt). A characterization of quotients of Abelian varieties by finite groups acting freely in codimension-one follows. This is achieved by tracing a correspondence between \emph{polystable} (respectively \emph{stable}) reflexive sheaves with zero Chern classes and (irreducible) unitary representations of the fundamental group (see (\ref{invariant-HE}) and Theorem~\ref{NS-result}) that goes back to the celebrated results of Narasimhan-Seshadri~\cite{NS65}, Donaldson~\cite{Don87}, and Uhlenbeck-Yau~\cite{UY86} on stable holomorphic vector bundles:\\[-3mm]

\begin{changemargin}{0.5cm}{0.5cm}
\refstepcounter{equation}\theequation\label{NS}. \emph{On a compact K\"ahler manifold $X$ of dimension $n$ with a K\"ahler form $w$, a vector bundle $\sE$ is stable with vanishing first and second Chern classes, that is
 $$\chern_1(\sE)_{\bR}=0 \ \ \  \mbox{and} \   \   \int_X \chern_2(\sE)\wedge [w]^{n-2}=0,$$
if and only if it comes from an irreducible unitary representation of 
$\pi_1(X)$}.\\[-3mm]
\end{changemargin}

We recall that the notion of stability (in the sense of Mumford-Takemato) of a torsion-free coherent sheaf $\sF$ on such a manifold requires the \emph{slope} of $\sF$ with respect to $w$: $$\mu_w(\sF):=\int_X\frac{\chern_1(\sF)\wedge [w]^{n-1}}{\rank(\sF)}.$$ We say that a torsion-free coherent sheaf $\sE$ is $[w]$-(\emph{semi})\emph{stable} if the inequality 
\begin{equation}\label{stability}
  \mu_w(\sG)<\mu_w(\sE)\ \ \ \ 
  (\mbox{respectively}\ \mu_w(\sG)\le\mu_w(\sE))
    \end{equation}
holds for every (torsion-free) coherent subsheaf $\sG$ of $\sE$ with
$0<\rank(\sG)<\rank(\sE)$. 
This notion generalizes in the projective category to the case 
when $X$ is an $n$-dimensional normal projective
variety with $n-1$ ample divisors $H_1,\dots,H_{n-1}$. In this case, 
a torsion-free coherent sheaf $\sE$ is said to be (semi)stable with respect to the polarization 
$\hbar :=(H_1,\ldots,H_{n-1})$ if (\ref{stability}) holds with 
the slope of a subsheaf $\sF$ above replaced by 
$$\mu_{\hbar}(\sF):=\chern_1(\sF)\cdot H_1\cdot \ldots\cdot H_{n-1}
/{\rank(\sF)}.$$
Note that this is well defined since 
$\chern_1(\sF)=\chern_1(\det \sF)$ and $\det \sF$ is invertible
outside the singular locus of $X$, which is of codimension two
or more. We say that a reflexive (or torsion-free) coherent sheaf $\sE$ 
on $X$ is \emph{generically semi-positive} if all its torsion-free quotients $\sF$ have semipositive slopes
$\mu_{\hbar}(\sF)$ with respect to every polarization 
$\hbar :=(H_1,\ldots,H_{n-1})$ with all $H_i$ ample.
We note that if $\det \sE$ is numerically trivial, then this
condition is equivalent to the semistability of $\sE$ (or of the dual of 
$\sE$) with respect to all such polarizations. Also, in the case $\sE$ is generically semi-positive or generically semi-negative and 
$\det \sE$ is $\bQ$-Cartier, then $\det \sE$ (or equivalently $\chern_1 (\sE)$) is numerically trivial if and only if $\chern_1(\sE)\cdot 
H_1\cdot \ldots\cdot H_{n-1}=0$ for some $(n-1)$-tuple of ample divisors $(H_1\cdot \ldots\cdot H_{n-1})$, see Lemma~\ref{triviality:2}. An important example of a generically semi-positive 
sheaf is the cotangent sheaf of a non-uniruled normal projective variety  \cite{Miy85}. \\[-2mm]

Throughout this paper, we work with a normal \emph{complex} projective variety $X$ having an orbifold structure (i.e. having only quotient singularities) 
in codimension-$2$ (see Section~\ref{Q-stuff}). This means 
that, if we cut down $X$ by $(n-2)$ very ample divisors, the general resulting surface  has only isolated quotient singularities, and hence inherits an orbifold structure (or $\bQ$-structure). Hence, there is a well-defined intersection pairing between $\bQ$-Chern classes $\widehat \chern_2(\sF)$, $\widehat \chern_1^2(\sF)$ and $(n-2)$-tuples of ample divisors on $X$ and similarly for the second $\bQ$-Chern character 
${\widehat\ch_2}(\sF):= (\widehat \chern_1^2(\sF) - 2\widehat \chern_2)/2$, see section~\ref{prelim}. 
The same holds for the intersection pairing between the first $\bQ$-Chern class ($\bQ$-Chern character) $\widehat \chern_1(\sF)(={\widehat\ch_1}(\sF))$ and $(n-1)$-tuples of ample divisors.
In this context, we have the following 
analog of (\ref{NS}) for a normal projective variety $X$  (see Section~\ref{KH} 
for its proof). 

\begin{thm}\label{NS-result} 
Let $X$ be an $n$-dimensional klt projective variety, $X_{\reg}$ its nonsingular locus and $\sF$ a reflexive coherent sheaf on $X$. Then (the analytification of) $\sF|_{X_{\reg}}$ comes from an irreducible unitary representation (respectively, an unitary representation) of $\pi_1(X_{\reg})$ if and only if, for some (and in fact for all) ample divisors 
$H_1,\ldots,H_{n-1}$ on $X$, we have:
\begin{enumerate}\label{assumps}
\item \label{assume:1} The reflexive sheaf $\sF$ is stable (respectively, polystable) with respect to the polarization $\hbar:=(H_1,\ldots,H_{n-1})$.\\[-3mm]
\item \label{assume:2} The first and second $\bQ$-Chern characters of $\sF$ verify the vanishing conditions:
\begin{flalign*}
& \widehat \ch_1(\sF)\cdot H_1\cdot \ldots \cdot H_{n-1}=0,\\
& \widehat \ch_2(\sF)\cdot H_1\cdot \ldots \cdot H_{n-2}=0.
\end{flalign*}
\end{enumerate}
In particular, $\sF$ is locally free on $X_{\reg}$ and is generically semi-positive in this case.
\end{thm}

Our main theorem below gives a characterization for finite quotients of Abelian varieties that are \'etale in codimension-$1$.

We recall that a \emph{klt space} is a normal $\mathbb{Q}$-Gorenstein space $X$ with at worst klt 
singularities (\cite{KM98}), i.e.\ $X$ admits a desingularization $\pi:\wtilde X \to X$ that 
satisfies  $a_i>-1$ for every $i$ in the ramification formula (we use $\simeq_{_{\Q}}$ to denote $\mathbb{Q}$-linear equivalence)
\begin{equation}\label{sings}
  K_{\wtilde X}\simeq_{_{\Q}}\pi^*(K_X)+\sum a_iE_i, 
     \end{equation}
where the $E_i$'s are the prime components of the exceptional divisor and $K_X$ is the canonical divisor of $X$ (up to linear equivalence) given by a $\bQ$-Cartier divisor defined by the sheaf isomorphism $\omega_X\cong \sO_X(K_X)$ with $\omega_X$ the dualizing sheaf of $X$. And we say that $X$ has at worst \emph{canonical} singularities, if we replace the inequality $a_i>-1$ above
by $a_i\geq 0$. We also recall that klt spaces have only quotient singularities in codimension-$2$ and that varieties with only quotient singularities are klt, see Remark~\ref{klt-orbifold}.\\[-3mm]

\begin{thm}
\label{main} Let $X$ be an $n$-dimensional 
normal projective variety. Then $X$ is a quotient 
of an Abelian variety by a finite action free in codimension-$1$ if and only if $X$ is klt and 

\begin{enumerate}\label{conditions:1}
\item\label{first} $K_X\equiv 0$.\\[-3mm]
\item\label{second} The second $\bQ$-Chern class of  $\sT_X:=(\Omega^1_{ X})^*$ respects the vanishing condition 
\begin{equation*}
\widehat \chern_2(\sT_X)\cdot A_1\cdot \ldots \cdot A_{n-2}=0,
\end{equation*}
for some $(n-2)$-tuple of ample divisors $(A_1,\ldots, A_{n-2})$.
\end{enumerate}

\end{thm}

\begin{rem}\label{AR} We remark that the two conditions (\ref{main}.\ref{first}) and (\ref{main}.\ref{second}) may be replaced by the following set of assumptions for some and hence for all polarization $\hbar=(H_1,\ldots,H_{n-1})$, with each $H_i$ ample (see Explanation~\ref{explaining}): \\[-3mm]

\begin{enumerate}\label{conditions:2}
\item\label{first0} The tangent sheaf $\sT_X$ is semistable with respect to $\hbar$ (automatically satisfied if the singularities of $X$ are canonical and Condition~(\ref{AR}.\ref{second0}) is satisfied).\\[-2mm]

\item \label{second0} The first and second $Q$-Chern characters of $\sT_X$ verify the vanishing conditions 
\begin{flalign*}
& \widehat \ch_1(\sT_X) \cdot H_1\cdot \ldots \cdot H_{n-1}=0,\\
&\widehat \ch_2(\sT_X)\cdot H_1\cdot \ldots \cdot H_{n-2}=0.\\[-4mm]
\end{flalign*}
\end{enumerate}
\end{rem}

Theorem~\ref{main} is a generalization of the classical uniformization theorem of S.T. Yau which states that a compact K\"ahler manifold $X$ of dimension $n$ with  K\"ahler class $[w]$ that satisfies $\chern_1(X)_{\bR}=0$ and $\int_X \chern_2(X)\wedge [w]^{n-2}=0$ is uniformized by $\mathbb{C}^n$ -- a consequence of Yau's solution to Calabi's conjecture~\cite{Yau78}, showing that a compact K\"ahler manifold with vanishing real first Chern class admits a Ricci-flat metric. The problem of extending this result to the singular setting  was proposed in a paper of Shepherd-Barron and Wilson~\cite{SBW94}. There, they show that  threefolds with  at most canonical singularities with numerically trivial first and second $\mathbb{Q}$-Chern classes are finite quotient of Abelian threefolds 
(unramified in codimension-$1$). Our basic strategies 
follow those of \cite{SBW94} and~\cite{GKP14a}. We take the natural classical approaches whenever possible. Theorem~\ref{main} for terminal varieties ($a_i>0$ in (\ref{sings})) and more generally for klt varieties that are smooth in codimension-$2$ have been established by Greb, Kebekus and Peternell~\cite{GKP14a} and sets the stage for this paper. We give two different proofs of the above theorem, the first via a result on the polystability of the tangent sheaf 
and the second 
via Miyaoka's theorem on the generic semipositivity of the cotangent sheaf of a non-uniruled variety by working out the following orbifold generalization of 
Simpson's correspondence between 
coherent sheaves with flat connections and semistable bundles  with zero first and second $\mathbb Q$-Chern characters. Both of these generalize to the case of other classical quotients, such as orbifold quotients of the ball, that we will treat elsewhere.

\begin{thm}
\label{finite-resolution} Let $X$ be a klt projective variety, 
$\hbar:=(H_1,\ldots,H_{n-1})$ a polarization on $X$ and $\sE$
a coherent reflexive sheaf on $X$. Then $\sE|_{X_{\reg}}$ is locally free and flat, i.e. given by a representation of $\pi_1(X_{\reg})$, if $\sE$ is semistable with respect to $\hbar$ and verifies 
\begin{flalign*}
& \widehat \ch_1(\sE)\cdot H_1\cdot \ldots \cdot H_{n-1}=0,\\
 & \widehat \ch_2(\sE)\cdot H_1\cdot \ldots \cdot H_{n-2}=0.
\end{flalign*}
Furthermore, there exists a finite, Galois morphism $f:Y\to X$ \'etale over $X_{\reg}$, independent of $\sE$, such that $(f^*\sE)^{**}$ is locally-free, equivariantly flat and with numerically trivial determinant over $Y$.

\end{thm}

It goes without saying that Yau's resolution of the Calabi conjecture and the Donaldson-Uhlenbeck-Yau theorem are basic ingredients in our proof of Theorem~\ref{main}. 
Key to our proof also include the  resolution of the Zariski-Lipman conjecture for klt spaces due to Greb, Kebekus, Kov\'acs and Peternel~\cite[Thm.~6.1]{GKKP11} and the  result in~\cite[Thm.~1.13]{GKP14a} of Greb, Kebekus and Peternell regarding the extension of flat vector bundles across the singular locus of a klt variety up to a finite cover obtained via Chenyang Xu's result on the finiteness of the local algebraic fundamental groups of klt varieties \cite{Xu12}. These latter recent 
results are used only to show that the singularities arising in 
Theorem~\ref{main} are quotient singularities and are not used
for Theorem~\ref{NS-result} and the first part of Theorem~\ref{finite-resolution}, which are results of independent interests even in dimension two. Compared to past approaches, crucial new  ingredients include the use, for a finite group $G$ acting on a normal projective variety $X$, of the $G$-equivariant version of the Donaldson-Uhlenbeck-Yau correspondance coupled with an equivariant resolution of $X$, of
Langer's  restriction theorems for (semi)stable sheaves, of Jordan-H\"older filtrations for $G$-semistable sheaves by $G$-equivariant subsheaves.


\begin{rem}[Comparing $\bQ$-Chern classes of $X$ to Chern classes of smooth models of $X$] Let $X$ be a minimal, projective klt variety and $\pi:\wtilde X\to X$ a resolution of $X$. The Chern class calculations in~\cite[Prop.~1.1]{SBW94}, together with the fact that the exceptional locus of a resolution of surface klt singularities is given by a tree of $\bP^1$'s, show that the inequality 
$$\widehat \chern_2(\sT_X)\cdot H^{n-2}\leq \chern_2(\sT_{\wtilde X})\cdot \pi^*H^{n-2}$$ 
holds for all ample divisor $H$ on $X$ and that the inequality is strict as long as $X$ is \emph{not} smooth in codimension-two. In particular if 
\begin{equation}\label{eq:compare}
K_X\equiv 0  \  \  \   ,  \ \ \  \chern_2(\sT_{\wtilde X})\cdot \pi^*H^{n-2}=0,
\end{equation}
then from Theorem~\ref{main} we find that $X$ is a finite quotient of an Abelian variety (notice that, by Proposition~\ref{SS-KLT}, $\sT_X$ is semistable in this case so that the Bogomolov inequality (\ref{kaw-bog}) gives  
$\widehat \chern_2(\sT_X)\cdot H^{n-2}\geq 0$). This uniformization result via the Chern class conditions 
(\ref{eq:compare}) in the case $X$ has at most canonical singularities has been established in~\cite[Thm.~1.4]{GKP14b} using a stability theory for sheaves with respect to movable classes. 

\end{rem}

\subsection{Acknowledgements}
 
The authors owe a debt of gratitude to Stefan Kebekus for fruitful conversations leading to a strengthening of results in the first draft of this paper and for his kind invitation of and hospitality during the first author's short visit to Freiburg where these took place. We also thank Chenyang Xu for an answer to a pertinent question related to the paper. A special thanks is owed to the anonymous referee for pointing out 
some mistakes in an earlier draft of this paper and for the helpful suggestions with regards to 
the presentation. 

\section{Preliminaries on $Q$-structures and $Q$-Chern classes}\label{prelim}
In this section, we first give a very brief overview of the theory of $\bQ$-vector bundles on an orbifold $X$ (or Satake's $V$-bundles on a
$V$-manifold~\cite{Sat56}). Then we provide, via the Bogomolov inequality for the $\bQ$-Chern classes of semistable $\bQ$-vector bundles, a numerical  criterion for $\bQ$-Chern classes of generically semi-positive reflexive sheaves  to vanish (Section~\ref{criterion:triv}). Finally we collect some basic facts on the behaviour of reflexive sheaves under a natural class of finite surjective maps between normal varieties. 

\subsection{Reflexive operations} 

In this paper, all objects are defined over $\bC$, all coherent sheaf on an algebraic (or analytic) variety $X$ are $\sO_X$-modules as well as  all torsion-free or reflexive sheaves on $X$. We denote the reflexive hull of a coherent sheaf $\sF$ of rank $r$ by $\sF^{**}$ and define the reflexive exterior power by $\wedge^{[i]}\sF:=(\wedge^{i}\sF)^{**}$. In particular, $\det \sF$ is the reflexive hull $(\wedge^r\sF)^{**}$ of $\wedge^r\sF$. For a morphism $f:Y\to X$, the reflexive pull-back of a coherent sheaf $\sF$ on $X$ is denoted by $f^{[*]}\sF:=(f^*\sF)^{**}$. A useful fact about a reflexive sheave $\sE$ on a normal
variety $X$ is that it is locally free on an open subsets $X_0$ of $X_{\reg}$ with codimension $\geq 2$ complement in $X$ and that, for any such open subset, $\sE=i_*(\sE|_{X_0})$ where 
$i: X_0\into X$ is the inclusion. In particular, 
reflexive pullbacks behave well under composition of finite morphisms between normal varieties.
For an in-depth discussion we invite the reader to consult Hartshorne~\cite{Har80} and~\cite{GKKP11}.

\subsection{Local constructions}\label{Q-stuff}
For a reflexive sheaf over a complex analytic variety $X$ with only quotient singularities in codimension-$2$, we define the $\bQ$-Chern classes via metric Chern-forms analytically locally.  We then define the $\bQ$-Chern classes for an algebraic reflexive sheaf as those of its analytification.  We also discuss conditions that guarantee their vanishing considered as multilinear
forms on the N\'eron-Severi space when $X$ is projective.\\[-5mm]

\subsubsection{$\bQ$-vector bundles and $\bQ$-Chern classes} Let $X$ be a complex analytic variety \emph{with only quotient singularities}, i.e. locally analytically $X$ is a quotient of complex manifold by a finite group, acting freely in codimension-$1$. 
Let $\{U_{\alpha}\}$ be a finite cover of $X$ with local uniformizations, that is, for each $\alpha$ there exists a complex manifold $X_{\alpha}$ and a finite, proper, holomorphic map $p_{\alpha}:X_{\alpha}\to U_{\alpha}$ such that $U_{\alpha}=X_{\alpha} / G_{\alpha}$, where $G_{\alpha}=\Gal(X_{\alpha}/U_{\alpha})$. We call the $G_{\alpha}$-equivariant coherent sheaves on $X_{\alpha}$ $G_{\alpha}$-sheaves and their $G_{\alpha}$-equivariant subsheaves 
the $G_{\alpha}$-subsheaves  (see \cite{Mum83} or \cite{GKKP11}).
We call a coherent sheaf $\sE$ on $X$ a \emph{$\bQ$-vector bundle}, if for each $\alpha$, there exists a $G_{\alpha}$-locally-free sheaf $\sE_{\alpha}$ on $X_{\alpha}$ such that its sheaf of $G_{\alpha}$-invariants $\sE_{\alpha}^{G_{\alpha}}$ descends to $\sE|_{U_{\alpha}}$, i.e. $\sE_{\alpha}^{G_{\alpha}}=\sE|_{U_{\alpha}}$.\\[-4mm]

Now, let $h_{\alpha}$ be a collection of hermitian $G_{\alpha}$-invariant metrics on $\sE_{\alpha}$ verifying the natural compatibility conditions on overlaps. Such a collection exists by a partition of unity argument (subordinate to $\{U_{\alpha}\}$) since the $X_\alpha$'s are locally isomorphic over the $U_{\alpha}$ overlaps. Denote the $i$-th Chern form of $h_\alpha$ by $\Theta_{\alpha,i}(\sE,h_{\alpha})$. These forms are  $G_{\alpha}$-invariant by construction and naturally gives rise to $\bQ$-Chern forms $\Theta_i$ of $\sE$ (over $X_{\reg}$)  defined by the local data $p_{\alpha}^*(\Theta_i|_{U_{\alpha}})=\Theta_{\alpha,i}(\sE,h_{\alpha})$. These define natural cohomology classes $\widehat \chern_i(\sE):=[\Theta_i(\sE,h_{\alpha})]\in H^{2i}(X,\bQ)$ independent of the choice of the metrics $h_{\alpha}$ and are called the \emph{$\bQ$-Chern classes} (or orbifold-Chern classes) of the $\bQ$-bundle $\sE$, noticing that $V$-manifolds satisfy Poincar\'e duality with coefficients $\bQ$, cf.~\cite[Thm.~3]{Sat56}.
We refer to the original paper of Satake~\cite{Sat56} (see also~\cite[Sect.~2]{Kaw92})
for an in-depth discussion of these notions.  \\[-5mm]

\subsubsection{Reflexive sheaves as $\bQ$-vector bundles} A reflexive sheaf $\sE$ on a complex analytic normal surface with at worst quotient singularities has a natural $\bQ$-vector bundle structure defined by the locally free sheaves $\sE_{\alpha}:=p_{\alpha}^{[*]}(\sE|_{U_{\alpha}})$. More generally, let $X$ be any normal complex analytic variety with \emph{only quotient singularities in codimension-$2$}, that is, the maximal subvariety $X_1$ of $X$ with an orbifold structure, defined by removing the non-orbifold locus from $X$ has 
$\codim_X (X\setminus X_1)\geq 3$. Then
any reflexive sheaf on $X$ has a $\bQ$-vector bundle structure on an open subset $X_2\into X_1$ with
codimension-$3$ complement. In particular, we may define $\bQ$-Chern classes of a reflexive sheaf by restricting to $X_2$.\\[-4mm] 

Assume further that $X$ is projective. We define the 
$\bQ$-Chern classes of an algebraic reflexive sheaf $\sE$ restricted to $X_2$ to be those of the analytification $\sE^{\an}$ of $\sE$. Since $\codim_X (X\setminus X_2)\geq 3$, there are well defined intersection pairings between $\widehat \chern_1(\sE)$ and $(n-1)$-tuples of ample divisors $(H_1,\ldots,H_{n-1})$ and between
any linear combination $\Delta:=a\cdot \widehat \chern_2(\sE)+b\cdot \widehat \chern_1^{\,2}(\sE)$ and the $(n-2)$-tuples $(H_1,\ldots,H_{n-2})$, the latter pairing via 
\begin{equation}\label{lin-comb}
\Delta\cdot H_1\cdot\ldots\cdot H_{n-2}:=\Delta|_{X_2}\cdot H_1\cdot\ldots \cdot H_{n-2}.
\end{equation}

\begin{rem}\label{klt-orbifold} We recall that klt spaces have only quotient singularities in codimension-$2$, 
cf.~\cite[Prop.~9.4]{GKKP11}. Therefore our discussion above is valid for any reflexive sheaf over a projective klt variety. Also varieties with only quotient singularities are klt  from the fact that a finite morphism $f:Y\to X$ \'etale in codimension-$1$ between normal varieties preserve the klt condition by \cite[Prop. 5.20]{KM98}. In particular, if $X$ as given in this previous sentence has only quotient singularities in codimension-$1$, then so does $Y$ (since $X_1$ is klt in this case).
\end{rem}

\begin{rem}\label{Q-chern-S}
Let $X$ and $\sE$ be as before. Define $S:=D_1\cap \cdots \cap D_{n-2}$ to be the orbifold projective surface in $X_2$ cut out by general members $D_i$ of basepoint-free linear system $|m\cdot H_i|$ for $m$ being a sufficiently large positive integer. 
We recall that $S$ may be chosen such that $\sE|_S$ is reflexive, see~\cite[Cor.~1.1.15]{HL10} for the smooth case and~\cite[Thm.~12.2.1]{Gro65} for the general setting. 
Therefore $\sE|_S$ has a natural $\bQ$-vector bundle structure (one can also see this by simply restricting the $\bQ$-vector bundle structure that is already enjoyed by $\sE|_{X_2}$ to the general surface $S\into X_2$). We may interpret the intersection number in  (\ref{lin-comb}) as the rational number $(1/{m^{n-1}})\cdot \Delta(\sE|_S)$ (we note that this number is obtained by integrating over the fundamental class of $S$ and is independent of the choice of $S$ for fixed  $H_1,\dots H_{n-2}$ and $m$. Thus, we may and will understand $\widehat \chern_2(\sE)$ (and similarly $\widehat \chern^{\,2}_1(\sE)$), following \cite{SBW94}, as multilinear forms on the N\'eron-Severi space $\mathrm{NS}(X)_{\bQ}$ (see below).

\end{rem}

\subsubsection{Numerical triviality of $\bQ$-Chern classes on the Picard group.}

\begin{defn}[Numerical Triviality of $\bQ$-Chern Classes of Reflexive Sheaves]\label{triviality} Let $X$ be a normal projective variety with only quotient singularities in codimension-$2$. For a reflexive sheaf $\sE$ on $X$, we say that the first and second $\bQ$-Chern classes of $\sE$ are numerically trivial on the Picard group (or simply trivial on $X$), and we write $\widehat \chern_i(\sE)\equiv_{_{ns}} 0$, $i=1,2$, if $\widehat \chern_i(\sE)$ defines a vanishing multilinear form on $\mathrm{NS}(X)_{\bQ}$. Since the $\bR$-span of the ample classes is open in 
$\mathrm{NS}(X)_{\bR}$, we have
\begin{flalign*}\label{triv}
& \widehat \chern_1(\sE)\equiv_{_{ns}} 0 \iff \widehat \chern_1(\sE)\cdot H_1\cdot \ldots\cdot H_{n-1}=0\ \ \forall \ (H_1,\ldots,H_{n-1}),\\
& \widehat \chern_2(\sE)\equiv_{_{ns}} 0 \iff \widehat \chern_2(\sE)\cdot H_1\cdot \ldots \cdot H_{n-2}=0\ \ \forall \  (H_1,\ldots,H_{n-2}),
\end{flalign*}
where $H_1,\ldots,H_{n-1}$ are ample divisor on $X$.
\end{defn}

In the case $\det \sE$ is $\bQ$-Cartier, it is an elementary 
exercise to show that $\widehat \chern_1(\sE)\equiv_{_{ns}} 0$
if and only if $\widehat \chern_1(\sE)\equiv 0$. One can also
appeal directly to the numerical triviality criterion of Kleiman 
given in \cite{Kle66} to see this. When it is not
$\bQ$-Cartier however, it makes little sense to talk about
numerical triviality in the usual sense since intersection with
arbitrary curves has no sense and, even if it does, this notion
would be different from the numerical triviality on the Picard group 
in general.
\

\subsubsection{Bogomolov inequality for $\mathbb{Q}$-bundles} 

With the setup as above and $\rank(\sE)= r$, a natural combination of $\bQ$-Chern classes as defined in (\ref{lin-comb}) is $$\Delta_B({\sE}):=2r\cdot \widehat \chern_2(\sE)-(r-1)\cdot {\widehat c_1}^{\,2}(\sE).$$ 
According to~\cite[Lem.~2.5]{Kaw92} any semistable reflexive sheaf $\sF$ on a \emph{projective} normal surface $S$ with only quotient singularities verifies the \emph{Bogomolov inequality}
\begin{equation}\label{kaw-bog}
\Delta_B(\sF)\geq 0.
\end{equation}

Now assume that $\sE$ is semistable with respect to a polarization $(H_1,\ldots,H_{n-1})$. Then, by the classical result of Mehta-Ramanthan~\cite{MR82} generalized by Flenner~\cite{Fl84}, the restriction $\sE|_S$ is also semistable, where $S:=D_1\cap \ldots \cap\widehat D_i \cap \ldots \cap D_{n-1}$ is the complete intersection surface cut out by general members $D_i\in |m\cdot H_i|$, $m\gg0$ and $i\in\{1,\ldots,n-1\}$, after removing an ample divisor $H_i$ from the $(n-1)$-tuple $(H_1,\ldots, H_{n-1})$. Therefore, thanks to the Bogomolov inequality (\ref{kaw-bog}), we have 
\begin{equation}\label{Bogy}
\Delta_B(\sE)\cdot H_1\cdot \ldots \cdot \widehat H_i \ldots\cdot H_{n-1}\geq 0 \ \ \forall i,
\end{equation}
where $(H_1,\ldots,\widehat H_i,\ldots,H_{n-1})$ is the $(n-2)$-tuple of ample divisors defined by removing the ample divisor $H_i$ from $(H_1,\ldots, H_{n-1})$ (see Remark~\ref{Q-chern-S} for the definition of the intersection in inequality (\ref{Bogy}).)  \\[-2mm]

\subsubsection{Numerical triviality criterion for $\bQ$-Chern classes}\label{criterion:triv}  Generically semi-positive reflexive sheaves (over normal varieties with only quotient singularities in codimension-$2$) are central objects of this paper. In the next lemma we show that the $\bQ$-Chern classes of such sheaves verify a natural numerical triviality criterion (on the Picard group). We note that for minimal varieties, the tangent 
sheaf is, thanks to Miyaoka's result, an example of a generically semi-positive sheaf and  
in this context the vanishing of $c_2$ in the following Lemma has already been proved
in~\cite[Prop.~4.8]{GKP14a}.

\

\begin{lem}
\label{triviality:2} Let $X$ be a normal projective variety $X$ with only quotient singularities in codimension-$2$ and $\sE$ a generically semi-positive reflexive sheaf on $X$. Assume that 
\begin{equation}\label{num-triv}
  \widehat \chern_1(\sE)\cdot H_1\cdot \ldots\cdot H_{n-1}=0
     \end{equation}
holds for some $(n-1)$-tuple of ample divisors $(H_1,\ldots,H_{n-1})$, then $\widehat \chern_1(\sE) \equiv_{_{ns}} 0$.\\ 
\indent If we assume furthermore that $$\widehat \chern_2(\sE)\cdot H_1'\cdot \ldots \cdot H_{n-2}'=0$$ for an $(n-2)$-tuple of ample divisors $(H_1',\ldots,H_{n-2}')$, then 
$\widehat \chern_2(\sE)\equiv_{_{ns}} 0$.

\end{lem}

\begin{proof} Aiming for a contradiction, suppose there exists a polarization $(A_1,\ldots,A_{n-1})$ such that $\widehat \chern_1(\sE)\cdot A_1\cdot \ldots\cdot A_{n-1}\neq 0$. Then, by the generic semi-positivity assumption, we have
\begin{equation}\label{eq:-2}
  \widehat \chern_1(\sE)\cdot A_1\cdot \ldots\cdot A_{n-1}>0.
    \end{equation}
Now let $m\in \mathbb{N}^+$ be a sufficiently large integer such that $(mH_i-A_i)$ is ample for all $i\in\{1,\ldots,n-1\}$ and consider the equality 
\begin{IEEEeqnarray}{rCl}\label{eq:-1}
m^{n-1}\cdot \widehat \chern_1(\sE)\cdot H_1\cdot \ldots \cdot H_{n-1} & = & \widehat \chern_1(\sE)\cdot \bigl((mH_1-A_1)+A_1\bigr)\cdot \nonumber\\
&&  \cdot \> \ldots \cdot \bigl((mH_{n-1}-A_{n-1})+A_{n-1}\bigr).
\end{IEEEeqnarray}
The right-hand side of (\ref{eq:-1}) is strictly positive by (\ref{eq:-2}) (and by the generic semi-positivity of $\sE$). But the left hand-side is equal to zero by the assumption, a contradiction.   \\[-3mm]
   
To prove the numerical triviality of $\widehat \chern_2(\sE)$, we argue similarly by using the Bogomolov inequality (\ref{Bogy}): First we observe that the generic semi-positivity of $\sE$ together with $\widehat \chern_1(\sE) \equiv_{_{ns}} 0$ implies that $\sE$ is semistable independent of the choice of a polarization. Therefore the second $\mathbb{Q}$-Chern class of $\sE$ is "pseudo-effective" by (\ref{Bogy}), that is
\begin{equation}\label{eq:PE}
   \widehat \chern_2(\sE)\cdot H_1\cdot \ldots\cdot H_{n-2}\geq 0
     \end{equation}
holds for all $(n-2)$-tuples of ample divisors $(H_1,\ldots, H_{n-2})$.  Suppose to the contrary that there exists  ample divisors $A_1',\ldots,A_{n-2}'$ such that $\widehat \chern_2(\sE)\cdot A_1'\cdot \ldots\cdot A_{n-2}'\neq 0$, i.e. 
\begin{equation}\label{positive}
   \widehat \chern_2(\sE)\cdot A_1'\cdot \ldots\cdot A_{n-2}'>0.
      \end{equation}
Set $m'\in \mathbb{N}^+$ to be a sufficiently large positive integer such that $(m'H_i'-A_i')$ is ample for all $i\in \{1,\ldots,n-2\}$. Then
\begin{IEEEeqnarray*}{rCl}
0=m'^{n-2}\cdot \widehat \chern_2(\sE)\cdot H_1'\cdot \ldots \cdot H_{n-2}' & = & \widehat \chern_2(\sE)\cdot \bigl((mH_1'-A_1')+A_1'\bigr)\cdot \nonumber\\
&& \cdot \> \ldots \cdot \bigl((mH_{n-2}'-A_{n-2}')+A_{n-2}'\bigr),
\end{IEEEeqnarray*}
which is a contradiction as the right-hand side is strictly positive by the Bogomolov inequality~\ref{Bogy} and the inequality~\ref{positive}.
\end{proof}
\noindent\\[-11mm]
\begin{rem}\label{equiv-vanish} Lemma~\ref{triviality:2} (see also the discussion after Definition~\ref{triviality}) in particular shows that for generically semipositive reflexive sheaves $\sE$ whose determinant 
$(\det \sE)$ is a $\bQ$-Cartier divisor, the two sets of vanishing conditions [$\det (\sE)\equiv 0$, $\widehat \chern_2(\sE)\equiv_{_{ns}}0$] and [$\widehat \chern_1(\sE)\cdot A_1\cdot \ldots \cdot A_{n-1}=0$, $\widehat \ch_2(\sE)\cdot H_1\ldots \cdot H_{n-2}=0$], where $(A_1,\ldots, A_{n-1})$ and $(H_1,\ldots, H_{n-2})$ are any $(n-1)$ and $(n-2)$-tuples of ample divisors in $X$, are equivalent. So for example in the case of a non-uniruled $\bQ$-Gorenstein variety $X$, the vanishing conditions in (\ref{assumps}.\ref{assume:2}) when $\sF=\sT_X$ is the same as $K_X\equiv 0$ and $\widehat \chern_2(\sT_X)\equiv_{_{ns}}0$.\\[-3mm]
\end{rem}

\subsection{$\bQ$-sheaves and global constructions}\label{global}

The theory of $\bQ$-sheaves was introduced by Mumford~\cite{Mum83} as an algebraic generalization of $\bQ$-vector bundles to a much larger class of coherent sheaves for which a meaningful notion of Chern classes can be defined. In this section we briefly recall some elementary facts that are needed for our results and we refer to Mumford~\cite{Mum83} (see also~\cite{Meg92}) for a detailed account of this theory. 
A concise but detailed account of all the notions and results needed in this 
paper can also be found in~\cite[Sect.~3]{GKPT15}.

Let $X$ be a normal projective variety with only quotient singularities. Then according to Mumford (\cite[Chapt.~2]{Mum83}) there exists a \emph{$\bQ$-structure}  given by the collection of charts $(U_{\alpha},p_{\alpha}:X_{\alpha}\to U_{\alpha})$, where $U_{\alpha}$ are quasi-projective, $p_{\alpha}$ are \'etale in codimension-$1$ and $X_{\alpha}$ are smooth. Let $G_{\alpha}:=\Gal(X_{\alpha}/U_{\alpha})$. We call a coherent sheaf $\sE$ on $X$ a \emph{$\bQ$-sheaf}, if there exists coherent $G_{\alpha}$-sheaves $\sE_{\alpha}$ on $X_{\alpha}$ such that $\sE_{\alpha}^{G_{\alpha}}=\sE|_{U_{\alpha}}$.

Now, let $K$ be a Galois extension of the function field $k(X)$ containing all the function fields $k(X_{\alpha})$ and let $\widehat X$ be the normalization of $X$ in $K$. Let $G$ be the Galois group. By construction, the corresponding finite morphism $p:\widehat X\to X$ factors though each $p_{\alpha}:X_{\alpha}\to U_{\alpha}$, i.e. there exists a collection of  finite morphisms $q_{\alpha}:\widehat X_{\alpha}\to X_{\alpha}$ giving a commutative diagram \\[-7mm]
$$
 \begin{xymatrix}{
    \widehat X_{\alpha}  \ar@/^5mm/[rrrr]^{p|_{\widehat X_{\alpha}}} \ar[rr]_{q_{\alpha}} && X_{\alpha}
     \ar[rr]_{p_{\alpha}} && U_{\alpha}.}
    \end{xymatrix} 
 $$ 
For a $\bQ$-sheaf (or a $\bQ$-vector bundle) $\sE$ on $X$, we can define a coherent sheaf $\widehat\sE$ on $\widehat X$ by gluing together the local data given by $\bigl(\widehat X_{\alpha},\widehat\sE_{\alpha}:=q_{\alpha}^*(\sE_{\alpha})\bigr)$. A result of Mumford (\cite[Prop.~2.1]{Mum83}) shows that when the global cover $\widehat X$ is Cohen-Macaulay,  any coherent sheaf $\widehat \sE$ on $\widehat X$ 
admits finite resolutions by locally-free sheaves and hence admits well-defined Chern classes. In our situation, we define the i-th $\bQ$-Chern class of $\sE$ as a ($G$-invariant) cohomology class on $\widehat X$ by $\widehat \chern_i(\sE):=(1\big{/} |G|)\chern_i(\widehat\sE)$. Following \cite{Mum83}, the $G$-invariant cohomology classes on $\widehat X$ identifies with homology cycles of complementary dimension on $X$ and
we define the intersection of $\widehat \chern_i(\sE)$ with cycles on $X$ by its intersection with the pullback cycle on $\widehat X$. This definition agrees with the analytic one in Section~\ref{Q-stuff} via the projection formula.  We refer the reader to~\cite{Mum83} for the intersection theory of $\bQ$-sheaves and note in particular that, as any normal surface is Cohen-Macaulay, we can always define $\bQ$-Chern classes of $\bQ$-sheaves on a normal irreducible surface $S$ with only quotient singularities. Finally, we note that the algebraic Hodge-Index theorem holds for $\widehat\chern_1$ by considering it on a desingularization of $\widehat S$.

\begin{rem}[An Equivalent Definition for $\widehat \chern_1\equiv 0$]\label{equi-def} Let $S$ be a normal surface with only quotient singularities. Let $p:\widehat S\to S$ be the global cover that was constructed above and $\widehat\sE$ the $G$-locally-free sheaf on $\widehat S$. It is not difficult to see that 
\begin{equation}\label{equivalence}
\widehat \chern_1(\sE)\equiv 0 \iff \chern_1(\widehat\sE)\equiv 0.
\end{equation}
The reason is that the $\bQ$-factoriality of $S$ (recall that any normal variety with quotient singularities is $\bQ$-factorial~\cite[Prop.~5.15]{KM98}), together with $\widehat \chern_1(\sE)\equiv 0$, implies that $\det(\sE)$ is numerically trivial as a $\bQ$-Cartier divisor. Notice that by construction we have $\widehat\sE=p^{[*]}(\sE)$. Therefore, the projection formula for Chern classes of pull-back bundles implies that $\chern_1(\widehat\sE)\cdot A=0$, for every ample divisor $A$ in $\widehat S$. The equivalence (\ref{equivalence}) now follows from Kleiman's numerical triviality criterion for $\bQ$-Cartier divisors~\cite[Prop.~3]{Kle66} which gives the elementary (linear algebra) fact that over a normal projective variety a $\bQ$-Cartier divisor is numerically trivial if it has zero intersection with all polarizations.\\[-3mm]
\end{rem}

\subsection{Behaviour of reflexive sheaves under finite quasi-\'etale morphisms}

We offer some elementary facts on the behaviour of reflexive sheaves under finite morphisms that are \'etale in codimension-$1$, i.e.\ finite \emph{quasi-\'etale} morphisms, ending with an observation regarding the stability of the tangent sheaf of klt varieties. 

\begin{lem}
\label{Q-projection} Let $X$ be a normal projective variety with only quotient singularities in codimension-$2$. Let $f:Y\to X$ be a finite Galois morphism that is \'etale in codimension-$1$ (with $Y$ normal). Then $Y$ has only quotient singularities in codimension-$2$ and,
for $\Delta$ a linear combination of $\widehat \chern_2$ and $\widehat \chern_1^2$,
a reflexive sheaf $\sE$ on $X$ satisfies $\Delta(\sE)\cdot H_1\cdot \ldots \cdot H_{n-2}=0$ for some $(n-2)$-tuple of ample divisors $(H_1,\ldots,H_{n-2})$ if and only if 
\begin{equation*}
\Delta\bigl(f^{[*]}(\sE)\bigr)\cdot f^*H_1\cdot \ldots \cdot f^*H_{n-2}=0.
\end{equation*}
 \\[-4mm]
\end{lem}

\begin{proof} The first claim follows from Remark~\ref{klt-orbifold}.
Now, define $\sG:=f^{[*]}(\sE)$ and let $(U_{\alpha},p_{\alpha}:X_{\alpha}\to U_{\alpha};U_{\alpha}=X_{\alpha}/G_{\alpha})$ be the local $\bQ$-structure (see Section~\ref{Q-stuff}) for $X_1$, where $X_1$ is equal to $X$ minus its non-orbifold locus. Let $\{\sE_{\alpha}\}$ be the collection of $G_{\alpha}$-sheaves on $X_{\alpha}$. Define $V_{\alpha}:=f^{-1}U_{\alpha}$ and let $Y_{\alpha}:=X_{\alpha}\times_{U_{\alpha}}V_{\alpha}$ be the fibre product given by the base change $p_{\alpha}:X_{\alpha}\to U_{\alpha}$ with the corresponding commutative diagram  
$$
  \xymatrix{
    Y_{\alpha} \ar[r]^{r_{\alpha}}  \ar[d]_{g_{\alpha}} & V_{\alpha}  \ar[d]^{f_{\alpha}:=f|_{V_{\alpha}}} \\
      X_{\alpha} \ar[r]^{P_{\alpha}}       & U_{\alpha}.
  }
  $$
Since $X_{\alpha}$ is smooth and $f_{\alpha}:V_{\alpha}\to U_{\alpha}$ is \'etale in codimension-$1$, so is $g_{\alpha}$. The purity of the branch locus says that $g_{\alpha}$ is \'etale and $Y_{\alpha}$ is smooth. Therefore $\sG_{\alpha}=g_{\alpha}^*(\sE_{\alpha})$, $\sG_{\alpha}$ being the locally-free sheaf on $Y_{\alpha}$ invariant under the action of $H_{\alpha}:=\mathrm{Gal}(Y_{\alpha}/ V_{\alpha})$ such that $\sG_{\alpha}^{H_{\alpha}}=\sG|_{V_\alpha}$. A collection of $G_{\alpha}$-invariant metrics $h_{\alpha}$ on $X_{\alpha}$ and  corresponding Chern forms $\Theta_{\alpha,i}(\sE,h_{\alpha})$ (see Section~\ref{Q-stuff}) thus induces $H_{\alpha}$-invariant Chern forms $\widehat\Theta_{\alpha,i}(\sG,g_{\alpha}^*h_{\alpha})$ on $Y_{\alpha}$ given by $g_{\alpha}^*(\Theta_{\alpha,i})$. Hence $\widehat \chern_i(\sG|_{Y_1})=(f|_{X_1})^*\widehat \chern_i(\sE|_{X_1})\in H^{2i}(Y_1,\bQ)$, where $Y_1:=f^{-1}(X_1)$. The lemma now follows from the projection formula.

\end{proof}

\begin{lem}\label{gsp} Let $f:Y\to X$ be a finite Galois morphism between normal varieties that is \'etale in codimenions-$1$, $\sE$ a reflexive sheaf on $X$
and $\hbar:=(H_1,\ldots,H_{n-1})$ a fixed polarization on $X$ with $H_i$ ample. Then $\sE$ is semistable with respect to $\hbar$ if and only if $f^{[*]}\sE$ is semistable with respect to $f^*\hbar:=(f^*H_1,\ldots,f^*H_{n-1})$.

\end{lem}

\begin{proof} 
The direction where $f^{[*]}\sE$ is assumed to be semistable is clear. 
For the other direction, the proof follows from the fact that the subsheaf $\widehat \sF$ of $f^{[*]}\sE$ with maximal $(f^*\hbar$)-slope is unique and thus invariant under the action of the group $\mathrm{Gal}(Y/X)$. As 
$f$ is \'etale in codimension-$1$, this implies that $\widehat \sF$ is a pull-back of a locally-free 
sheaf $\sF^\circ$ away from the branch locus $X_1$ of $f$. Let   
$\sF:=i_*(\sF^\circ)$, where $i$ is the inclusion $i: X\backslash X_1 \to X$, be the natural extension 
of $\sF^\circ$ onto $X$. As $\sE$ is reflexive, and since $\sF^\circ \subset \sE|_{X\backslash X_1}$, 
we have $\sF\subset \sE$. The claim now follows from the projection formula.

\end{proof}

Let $X$ be any non-uniruled normal projective variety.  The  generic semi-positivity result of Miyaoka~\cite{Miy85} says that the tangent sheaf of $X$ is generically semipositive. This is the case when a variety $X$ has numerically-trivial canonical divisor $K_X$ and $X$ has only canonical singularities. To see this, take $\pi:\wtilde X\to X$ to be a resolution of $X$ with the ramification formula $K_{\wtilde X}=\pi^*K_X+E$, where $E$ is, by definition, an effective exceptional divisor. Clearly, the numerical triviality of $K_X$ implies that $K_{\wtilde X}$ is  pseudo-effective. Therefore $\wtilde X$ is non-uniruled and thus so is $X$.
More generally, we have the following proposition.

\begin{prop}
\label{SS-KLT} Let $X$ be klt projective. If $K_X\equiv 0$, then $\Omega^{[1]}_X$ is generically semipositive.

\end{prop} 

\begin{proof} 
According to the abundance result of for klt varieties with numerically trivial canonical divisor
~\cite{Nak04}, we know that $K_X$ is a torsion, 
$\bQ$-Cartier divisor. Let $f:Z \to X$ be the associated index one cover, where 
$f^*(K_X)$ is trivial. By construction, $f$ is unramified in codimension one. Therefore 
$K_Z$ is also trivial. 
But $Z$ has at worst klt singularities by Remark~\ref{klt-orbifold}. Hence $Z$ has at most canonical singularities. From our discussion above, $Z$ is not uniruled. Thus $\Omega^{[1]}_Z=f^{[*]}\Omega^{[1]}_X$ is generically semipositive or 
equivalently $\sT_Z$ is generically seminegative. 
As $K_Z=0$, this implies that $\sT_Z$ is semistable with respect to 
$(f^*H_1, \ldots , f^*H_{n-1})$, for any $(n-1)$-tuple of ample divisors in $X$. Therefore by Lemma~\ref{gsp} 
we find that $\sT_X$ is semistable with respect to any $(H_1, \ldots, H_{n-1})$. Again, as $K_X\equiv 0$, this 
latter assertion is equivalent to $\Omega_X$ being generically semipositive. 
(Notice that in the argument above, it is possible to avoid using abundance results, once we 
have established Theorem~\ref{finite-resolution}. The reason is that Theorem~\ref{finite-resolution}
implies that there exists 
a finite Galois cover $f:Z\to X$ \'etale over $X_{\reg}$ which pulls back numerically trivial $\bQ$-Cartier divisors on $X$ to Cartier ones on $Z$, i.e. a simultaneous index one cover for such divisors.
This impies that $K_Z\simeq_{\bQ}f^*K_X$ is Cartier and this is all we need in to prove 
Proposition~\ref{SS-KLT}.)

\end{proof}

\section{Stable reflexive sheaves with vanishing $\mathbb{Q}$-Chern classes}\label{KH}


We recall that a hermitian metric $h$ on a holomorphic vector bundle $\sE$ over a compact K\"ahler manifold with K\"ahler form $w$ is said to satisfy the Einstein condition if 
\begin{equation}\label{HE}
i\, \Lambda_wF=\lambda  \id_{\sE}
\end{equation} 
for some $\lambda\in \mathbb{R}$ where $F$ is the curvature of the unitary connection compatible with the holomorphic structure. From the classical result of Donaldson, Uhlenbeck and Yau, we know that given a compact 
K\"ahler manifold $X$ of dimension $n$ and a K\"ahler class $[w]$, every $[w]$-stable holomorphic vector bundle $\sE$ admits a hermitian metric $h$ whose associated unitary connections is Hermitian-Einstein. Moreover if $\chern_1(\sE)_{\bR}=0$ and $$\int_X \chern_2(\sE)\wedge [w]^{n-2}=0,$$ then $(\sE,h)$ is flat. Our aim in this section is to prove a generalization of this result to the case of reflexive sheaves over normal projective varieties with only quotient singularities in codimension-$2$, namely Theorem~\ref{NS-result}. For this, we first examine as did \cite{SBW94} the question of how stability of a $\bQ$-vector bundle $\sE$ over a complex projective surface $S$ with only quotient singularities behaves under blowing-ups: Let $(U_{\alpha},\ p_{\alpha}\colon X_{\alpha}\to U_{\alpha})$ be the local $\bQ$-structure of $S$ (Section~\ref{Q-stuff}). Let $p:\widehat S\to S$ be the global finite cover with Galois group $G$ and let $\widehat\sE$ be the locally-free sheaf on $\widehat S$ such that $\widehat\sE^G=\sE$ (Section~\ref{global}). We study the stability of $\pi^*(\widehat\sE)$ on a $G$-equivariant desingularization $\pi:\wtilde S\to \widehat S$ (whose existence is guaranteed classically or by \cite{BM95} for example). Note that the actions of $G$ on $\widehat S$ and on $\widehat\sE$ lift uniquely to actions of $G$ on $\wtilde S$ and on $\pi^*(\widehat\sE)$ respectively. For expediency, we now allow all of our polarization divisors to be $Q$-Cartier.\\[-3mm]

\begin{prop}
\label{lift} Let $\sE$ be a reflexive coherent sheaf on a normal projective surface $S$ with only quotient singularities. Fix an ample divisor $H$ on $S$.  With the setup as above, if $\sE$ is stable with respect to $H$, then there exists a $G$-invariant polarization $\wtilde H$ on $\wtilde S$ such that $\mu_{\wtilde H}(\wtilde \sE)=\mu_{H}(\sE)$, where $\wtilde \sE:=\pi^*\widehat\sE$, and that $\wtilde \sE$ is $G$-stable with respect to $\wtilde H$, that is, for every 
 proper, $G$-equivariant subsheaf  $\ \wtilde \sG\subset \wtilde \sE$, we have the strict inequality\\[-5mm]
\begin{equation}\label{G-stability}
\mu_{\wtilde H}(\wtilde \sG)<\mu_{\wtilde H}(\wtilde \sE).\\[1mm]
\end{equation}
 \end{prop}

\begin{proof} First note that $\widehat\sE$ is semistable with respect to $\widehat H:=p^*(H)$. For otherwise the maximal destabilizing subsheaf of $\widehat\sE$ would descend to a proper $H$-destablizing subsheaf of $\sE$. 
This is because the maximal destablizing subsheaf is, thanks to its uniqueness, a $G$-sheaf and 
moreover as it is saturated inside $\widehat\sE$, it is given by a pull-back bundle outside codimension-$2$ 
(see~\cite[Prop.~2.16]{GKPT15} for a proof of this fact). It is also $G$-stable; otherwise there would exist a saturated $G$-equivariant semistable subsheaf $\widehat\sG\subset \widehat\sE$ of strictly smaller rank with
$\mu_{\widehat H}(\widehat\sG)= \mu_{\widehat H}(\widehat\sE),$
which would descend to a saturated semistable subsheaf 
$\sG\subsetneq \sE$ with $\mu_{H}(\sG)= \mu_{H}(\sE)$ and contradict the stability of $\sE$. Note also that one can arrange
 $\bigl(\pi^*(\widehat H)- E'\bigr)$ to be ample by a suitable choice of an effective and $\pi$-exceptional $\Q$-divisor $E'$ and, by averaging over $G$ if necessary, $G$-invariant.
Set $\wtilde H^\circ=\bigl(\pi^*(\widehat H)- E'\bigr)$ for this choice. Since $\widehat\sE$ is locally free, $\wtilde \sE=\pi^*\widehat\sE$ is trivial along the (reduced) exceptional divisor $E$ of $\pi$ and hence $\mu_{\wtilde H^\circ}(\wtilde \sE)=\mu_{\widehat H}(\widehat\sE)$.\\[-2mm]

Now, let $\wtilde \sF$ be any $G$-subsheaf of $\wtilde \sE$. Let $U$ be a Zariski-open subset of $\widehat S$
with $\codim_{\widehat S}(\widehat S\backslash U)=2$ such that 
$\pi|_{\pi^{-1}(U)}: \pi^{-1}(U)\to U$ is an isomorphism. Let $\widehat\sF^\circ:=(\pi^{-1}|_{\pi^{-1}U})^*\wtilde \sF$ be the $G$-subsheaf of $\widehat\sE|_U$ induced by $\wtilde \sF$. Define $\widehat\sF:={i_U}_*(\widehat\sF^\circ)$ to be the coherent extension of $\widehat\sF^\circ$ across $\widehat S\backslash U$ and notice that $\widehat\sF$ defines a $G$-equivariant coherent subsheaf of $\widehat\sE$. Now from the $G$-stability of $\widehat\sE$ with respect to $\widehat H$, we infer that the inequality 
\begin{equation}\label{ineq:stability1}
\mu_{\pi^*\widehat H}(\wtilde \sF)<\mu_{\pi^*\widehat H}(\wtilde \sE)
\end{equation}
holds. On the other hand it is easy to see that the set 
\begin{equation}\label{upperbound}
 \{\mu_{\wtilde H^\circ}(\wtilde \sG):\wtilde \sG \text{\ coherent subsheaf of $\wtilde \sE$}\}
 \end{equation}
admits an upper-bound. Therefore for all sufficiently small $\delta \in \bQ^+$,  we can always exploit the $G$-stability of $\widehat\sE$ (\ref{ineq:stability1}) to ensure that the inequality (\ref{G-stability}) holds by choosing  $\wtilde H:=\pi^*\widehat H+\delta \wtilde H^\circ$, i.e. $\wtilde \sE$ is $G$-stable with respect to $\wtilde H$.
\end{proof}

It is well known from the celebrated theorems of Donaldson and Uhlenbeck-Yau and from their proofs, see for example \cite[Thm.~5]{Pra93} and the remarks given therein, that given a compact K\"abler manifold equipped with the holomorphic action of a compact Lie group $G$ and a $G$-invariant K\"ahler form $w$, any holomorphic vector bundle that is $G$-stable with respect to $w$ admits a $G$-invariant Hermitian-Einstein connection. Therefore, 
as our choice of polarization $\wtilde H$ in the above proposition is $G$-invariant, the locally-free sheaf $\wtilde \sE=\pi^*{\widehat\sE}$ carries a Hermitian-Einstein connection $\wtilde D: {\wtilde \sE} \to\wtilde \sE \otimes  \Omega^1_{{\wtilde S}}$ that is $G$-invariant. If we further assume that $\widehat \ch_2(\sE)=0$ and $\mu_H(\sE)=0$, then $\ch_2(\wtilde \sE)=0$ and, with a K\"ahler form $w$ representing $\chern_1(\wtilde H)$, 
\begin{equation}\label{zero}
\frac{\sqrt{-1}}{2\pi}\int_{\wtilde S} \mbox{tr} (\Lambda_{w} F_{\wtilde D}) d \mbox{vol}(w)
=\int_{\wtilde S} \mbox{Ric}(\wtilde D)\wedge [w]=
\chern_1(\wtilde \sE)\cdot \wtilde H=0.
\end{equation}
In particular the unitary connection $\wtilde D$ is flat, i.e. $\wtilde D^2=0$. The flatness of $\wtilde D$ follows from 
the fact that $\wtilde D$ is Hermitian-Einstein (\ref{HE}), so that the vanishing condition (\ref{zero}) implies the vanishing of $\Lambda_{w} F_{\wtilde D}$, and from
the well-known Riemann bilinear identity, c.f. page 16 of \cite{Sim92} or equations~4.2 and 4.3 in Chapter IV of \cite{Kob87}, which takes the form
\begin{equation}\label{bilinear}
\int_{\wtilde S}\ch_2(\wtilde \sE)=C(||\Lambda_{w} F_{\wtilde D}||^2_{\mathrm{L}^2}-||F_{\wtilde D}||^2_{\mathrm{L}^2})
\end{equation}
for some positive constant $C$. \\
Furthermore, as $\wtilde D$ is compatible with the holomorphic structure of $\wtilde \sE$, 
its $(1,0)$ part defines an equivariant, holomorphically flat connection. 
Now, let $$\widehat D: \widehat \sE|_{\widehat S_{\reg}} \to \widehat \sE|_{\widehat S_{\reg}}\otimes (\Omega^1_{\widehat S_{\reg}})^G$$
be the induced $G$-invariant connection.
Here $(\Omega^1_{\widehat S_{\reg}})^G$ denotes the sheaf of $G$-invariant forms on $\widehat S_{\reg}$.
After taking $G$-invariant sections, $\widehat D$ induces a unitary, flat connection 
$D: \sE|_{S^\circ} \to \sE|_{S^\circ}\otimes \Omega^1_{S^\circ}$, where $S^\circ$ is $S_{\reg}$ 
minus a finite number of smooth points---the subset of $S_{\reg}$ over which 
we have $\widehat \sE|_{p^{-1}(S^\circ)}=p^*(\sE|_{S^\circ})$.
Since removing smooth points from $S_{\reg}$ does not change its fundamental group, we find 
that $\sE|_{S_{\reg}}$ is given by a unitary representation $\rho_{S}:\pi_1(S_{\reg})\to U(r,\bC)\subset GL(r,\mathbb{C})$ with $r$ the rank of $\sE$, cf.~\cite[Chapt.~I]{Kob87}. This representation is irreducible for otherwise we can extract a contradiction as follows: Let $C$ 
be a sufficiently general member of the basepoint-free linear system 
$|m\cdot H|$, $m\gg 0$, for which we have $C\subset S_{\reg}$ as 
a smooth projective curve and $\sE|_C$ is stable. To see this, let 
$\mu: T \to S$ be a resolution and notice that $\mu^{[*]}(\sE)$ is stable with respect to $\mu^*H$. 
If $\sE|_C$ is not stable, then $\mu^{[*]}(\sE)|_{\pi^{-1}C}$ 
is not stable, contradicting the 
Bogomolov Restriction Theorem, cf.~\cite{Bog79}.
Now, if the connection in $\sE|_{S_{\reg}}$ defined by $\rho_S$ 
is not irreducible, it restricts to a unitary but not irreducible connection 
in $\sE|_C$. By the theorem of Narasimhan-Seshadri, this implies that 
$\sE|_C$ is polystable, but not stable, contradicting the stability of $\sE|_C$.
For future references, we summarize this discussion as follows.\\[-3mm]

\

\begin{changemargin}{0.5cm}{0.5cm}
\refstepcounter{equation}\theequation\label{invariant-HE}. \emph{Given a normal irreducible surface $S$ with only quotient singularities, let $\sE$ be a stable (respectively polystable) reflexive sheaf over $S$. If\, $\widehat \chern_1(\sE)\cdot A=0$ for some ample divisor $A$ on $S$ and $\widehat \ch_2(\sE)=0$, then the analytification of $\sE|_{S_{\text{reg}}}$ is given by a unitary, irreducible (respectively possibly reducible) representation \\[-3mm]
$$\rho_S:\pi_1(S_{\reg})\to GL(r,\bC).$$}\\[-8mm]
\end{changemargin}

\

\subsection{General setup}
Since there is a significant overlap between the results and methods needed
for the proof of Theorems~\ref{NS-result} and~\ref{finite-resolution}, to 
avoid unnecessary repetitions, we collect some common ingredients in the 
following set-up.

\begin{setup}\label{setup} Let $X$ be a normal projective variety and $\sF$ 
any reflexive sheaf on $X$, stable with respect to $(H_1, \ldots, H_{n-1})$, 
with $H_i\subset X$ being ample. Let $\wtilde \pi: \wtilde X\to X$ be a  
resolution. According to~\cite[Thm.~5.2]{Lan04} the reflexive pull-back sheaf 
$\pi^{[*]}(\sF)$ satisfies a restriction Theorem, 
that is $\pi^{[*]}(\sF)|_{D_1}$ is stable, for every normal member $D_1$ of the 
linear system $|\pi^*(m\cdot H_1)|$, $m\gg 0$. It thus follows that 
$\sF$ also satisfies a restriction theorem. When $\sF$ is semistable, 
the semistability of the restriction (to general members of a linear system of 
very ample divisors) is due to Flenner, cf.~\cite{Fl84}. 

Now with the additional assumption that $X$ is klt, let $\mathcal V^\circ$ 
be the family of flat, locally-free, analytic sheaves on $X_{\reg}$. 
Define $\mathcal V_X^\circ=\{ \sF^\circ \}$ to be the family of 
locally-free, flat, analytic sheaves on $X_{\reg}$. Let $f: Y\to X$  be the
quasi-\'etale cover in~\cite[Thm.~1.5]{GKP14a}, where 
$\widehat \pi_1(Y)\cong \widehat \pi_1(Y_{\reg})$. As $(f|_{X_{\reg}})^*(\sF^\circ)$
is flat, according to~\cite[Thm.~1.14]{GKP14a} there exists 
a unique, locally-free, flat, algebraic sheaf $\wtilde \sF$ on $Y$ such that 
$(\wtilde \sF|_{Y_{\reg}})^{an}\cong (f|_{X_{\reg}})^*(\sF^\circ)$. 
Let $\mathcal V_Y=\{ \wtilde \sF \}$ denote the family of flat, locally-free, 
algebraic sheaves on $Y$ constructed as extensions of sheaves of 
form $(f|_{X_{\reg}})^*\sF^\circ$, where $\sF^\circ\in \mathcal V^\circ$.
Define $\mathcal V_X:= \{  (f_*(\wtilde \sF))^{**} \ \big| \ \wtilde \sF\in \mathcal V_Y \}$. 
As the family of locally-free, flat, algebraic sheaves on $Y$ is bounded, 
cf.~\cite[Prop.~9.1]{GKP14a}, $\mathcal V_X$ forms a bounded 
family of reflexive, algebraic sheaves on $X$.

Following the setup of the proof of~\cite[Thm.~1.20]{GKP14a}, 
we consider the increasing sequence of integers $m_1 \leq \ldots \leq m_{n-2}$
such that the complete intersection surface $S:=D_1\cap \ldots \cap D_{n-2}$, 
where $D_i$ is a general member of the basepoint-free linear system 
$|m_i\cdot H_i|$, verify the following properties. 

\begin{enumerate}\label{enum:properties}
\item \label{prop:1} The projective surface $S$ is normal, irreducible and klt, 
cf.~\cite[Lem.~5.17]{KM98}.

\item \label{prop:2} The restriction $\sF|_S$ is reflexive and (semi)stable with respect to 
$H_{n-1}|_S$, if $\sF$ is (semi)stable with respect to $(H_1, \ldots, H_{n-1})$. 
This is guaranteed by the results of 
Flenner and Langer (see the discussion in the beginning of the current setup).

\item \label{prop:3} We have the isomorphism $\pi_1(S_{\reg})\cong \pi_1(X_{\reg})$, 
which follows from Lefschetz-hyperplane theorem~\cite{HL85} for quasi-projective 
varieties. 

\item \label{prop:4} For any member $\sG\in \mathcal V_X$, the isomorphism $\sF\cong \sG$
holds, if and only if $\sF|_S \cong \sG|_S$, cf.~\cite[Cor.~5.3]{GKP14a}.
\end{enumerate}

\end{setup}

\
\subsection{Proof of Theorem~\ref{NS-result}} Observe that we only need to treat the stable case. This is because if $\sF=\oplus_i\sF_i$ is polystable, then, for each $i$, as $\mu_{\hbar}\sF_i=0$, the Hodge-index theorem implies that $\widehat \chern_1^2(\sF_i)\cdot H_1\cdot \ldots \cdot H_{n-2} \leq 0$ and so the Bogomolov inequality (\ref{Bogy}) 
\begin{equation*}
-2\cdot \widehat \ch_2(\sF_i)\cdot  H_1\cdot \ldots \cdot H_{n-2} \geq \frac{-1}{r_i}\cdot \widehat \chern_1^2(\sF_i)\cdot  H_1\cdot \ldots \cdot H_{n-2}\,\ \ (r_i:=\rank (\sF_i))
\end{equation*}
gives $\widehat \ch_2(\sF_i) \cdot H_1\cdot \ldots \cdot H_{n-2} \leq 0$ for all $i$. Hence, the additivity of the second $\bQ$-Chern character implies that every $\hbar$-stable component $\sF_i$ verifies $\widehat \ch_2(\sF_i)\cdot H_1\cdot \ldots \cdot H_{n-2}=0$. 
\\[-2mm]

We first assume that $(\sF|_{X_{\reg}})^{\an}$ is given by an irreducible unitary representation of $\pi_1(X_{\reg})$. Then, for every smooth irreducible curve $C$ cut out by general hyperplane sections
corresponding to high enough multiples of a polarization $\hbar$, the restriction $(\sF|_C)^{\an}$ also comes from an irreducible unitary representation of $\pi_1(C)$ via the surjectivity (by the Lefschez theorem) of the push-forward of the fundamental group induced by the inclusion $C\into X_{\reg}$. Now, according to the classical result of Narasimhan and Seshadri, the bundle $(\sF|_C)^{\an}$ is stable with degree zero. Therefore $\sF$ is stable of degree zero with respect to any polarization $\hbar$ (and hence is generically semi-positive). This gives one direction of the theorem.\\[-2mm]

To prove the reverse direction, let $S$ be the klt surface 
defined in Setup~\ref{setup}. As $\sF|_SS$ is a reflexive stable sheaf
(Property~\ref{enum:properties}.\ref{prop:2}). 
 Hence by (\ref{invariant-HE}), we find that $(\sF|_S)^{\an}|_{S_{\reg}}$ is defined by an irreducible unitary representation $\rho_{S}:\pi_1(S_{\reg})\to GL(r,\mathbb{C})$. On the other hand, the isomorphism 
 $\pi_1(S_{\reg})\cong \pi_1(X_{\reg})$ (Property~\ref{enum:properties}.\ref{prop:3})
 gives rise to an irreducible unitary representation $\rho:\pi_1(X_{\reg})\to GL(r,\mathbb{C})$. That is, there exists a locally-free, flat, analytic sheaf $\sG^\circ$ on $X_{\reg}$ (coming from an irreducible unitary representation) whose restriction to $S$ is isomorphic to $\sF_S^{\an}$:
\begin{equation}\label{iso:1}
   \sG^\circ|_{S_{\text{reg}}}\cong \sF|_{S_{\text{reg}}}.
       \end{equation}
Now let $\sG\in \mathcal V_X$ be the extension of $\sG^\circ$ 
in the sense of Setup~\ref{setup}. As $\sG$ and $\sF$ are both reflexive, 
the isomorphism~\ref{iso:1} implies that $\sG|_S \cong \sF|_S$.
Theorem~\ref{NS-result} now follows from Property~\ref{enum:properties}.\ref{prop:4}.
\hfill $\Box$\\[-2mm]

The next corollary is now an immediate consequence of Theorem~\ref{NS-result} and the result on extending flat sheaves 
across the singular locus of a klt variety after going to a suitable cover (where the contributions of the singularities to the algebraic fundamental group of the smooth locus disappear), cf.~\cite[Thm.~1.14]{GKP14a}. 

\begin{cor}[Desingularization of (Poly)Stable Reflexive Sheaves with Vanishing $\bQ$-Chern Classes Up to a Finite Quasi-\'etale Cover]\label{SirSimon} Let $X$ be a klt projective variety. There exists a finite Galois morphism $f:Y\to X$ \'etale over $X_{\reg}$ with Galois group $G$ such that $f^{[*]}\sF$ is a locally-free sheaf given by a $G$-equivariantly irreducible unitary representation of $\pi_1(Y)$ (respectively, a direct sum of such sheaves) for every reflexive sheaf $\sF$ on $X$ verifying the conditions (\ref{assumps}.\ref{assume:1}) and (\ref{assumps}.\ref{assume:2}) in Theorem~\ref{NS-result}.

This holds in particular for rank-one reflexive sheaves associated to $\mathbb Q$-Cartier divisors that are numerically equivalent to zero. \\[-3mm]
\end{cor}

\subsection{A proof of Theorem~\ref{main} via polystability}\label{polysection}\label{proof}
Let $g:Z \to X$ be the global index-$1$ cover provided in the 
last part of Corollary~\ref{SirSimon} (see also the proof of Proposition~\ref{SS-KLT}). Then $K_Z=g^{[*]}K_X$ is a numerically trivial Cartier divisor and $Z$ has only canonical singularities. According to the main result of \cite{GKP12}, there exists a quasi-\'etale cover $h:\widehat Z\to Z$ where $\sT_{\widehat Z}$ is polystable with respect to the polarization $h'^*(\hbar):=(h'^*(H_1),\ldots, h'^*(H_{n-1}))$ with $h':=h\circ g$. Since both $g$ and $h$ are unramified in codimension-$1$, we have the sheaf isomorphism $\det \sT_{\widehat Z}\cong \det(h'^{[*]}\sT_X)$ by the ramification formula. As a result, the natural inclusion of reflexive sheaves $\sT_{\widehat Z}\to h'^{[*]}\sT_X$ is an isomorphism.  On the other hand, as the $\bQ$-Chern classes behave well under quasi-\'etale morphisms (see Lemma~\ref{Q-projection}), the assumption $\widehat \chern_2(\sT_X)\cdot H_1\cdot \ldots\cdot H_{n-2}=0$ implies that 
\begin{equation*}
\widehat \chern_2(\sT_{\widehat Z})\cdot h'^*(H_1)\cdot \ldots\cdot h'^*(H_{n-2})=0.
\end{equation*}
Therefore by Corollary~\ref{SirSimon} we have a morphism $f:Y\to \widehat Z$ that is \'etale in codimension-$1$ such that $\sT_Y=f^{[*]}\sT_{\widehat Z}$ is locally-free (and flat). According to the resolution of Zariski-Lipman conjecture for klt spaces~\cite[Thm.~6.1]{GKKP11}, this implies that $Y$ is smooth. In particular $X$ has only quotient singularities. But again, as $f$ is \'etale in codimension-$1$, we have $K_Y\equiv 0$ and $\chern_2(\sT_Y)\cdot f^*(h'^*(H_1))\cdot \ldots \cdot f^*(h'^*(H_{n-2}))=0$.\\[-2mm]

The ``if" direction of Theorem~\ref{main} now follows from the classical uniformization result due to the fundamental work of Yau in \cite{Yau78} on the existence of a Ricci-flat metric in this case. See for example~\cite[\S IV.Cor.~4.15]{Kob87} or argue directly that the Ricci flat metric is actually flat using the Riemann bilinear relations (\ref{bilinear}) so that $\pi_1(Y)$ must act by isometry on the flat universal cover $\bC^n$ and must hence be an extension of a lattice in $\bC^n$ by a finite group.\\[-2mm]

To prove the ``only if" direction of Theorem~\ref{main}, notice that if $X$ is a finite quotient of an Abelian variety by a finite group acting freely in codimension-$1$, then it follows from the definition that $\widehat \chern_1(\sT_X)=\widehat \chern_2(\sT_X)=0$. We observe easily that $X$ is normal in this case and thus according to~\cite[Prop.~5.20]{KM98} that $X$ has at most klt singularities. \qed\\[-5mm]

\

\section{Semistable $\mathbb Q$-sheaves and their  correspondence with flat sheaves}\label{Simpson}

We give a proof of Theorem~\ref{finite-resolution} following the classical approach via the Jordan-H\"older filtration of a semistable sheaf, which also explicit a natural and necessary part of Simpson's proof of his celebrated correspondence. Theorem~\ref{main} is then a corollary of the local freeness result (up to a finite cover) of Theorem~\ref{finite-resolution} and the generic semipositivity theorem of the cotangent sheaf due to Miyaoka \cite{Miy85, Miy87}. 
We remark that our orbifold 
generalizations in Theorem~\ref{finite-resolution} of this 
important result of Simpson seem not to 
be previously known even in dimension $2$, 
when $X$ is an orbifold (for which the 
orbifold fundamental group is given by
$\pi_1(X_{\reg})$).

\subsection{Preparation for the Proof of Theorem~\ref{finite-resolution}}
If $\sE$ is $\hbar$-stable, then the result follows from Theorem~\ref{NS-result} and Corollary~\ref{SirSimon}.
In the semistable case we need to construct a Jordan-H\"older filtration for 
$\bQ$-sheaves over $\bQ$-surfaces. 

\begin{prop}[$\bQ$-Jordan H\"older filtration]\label{Q-JH} Let $S$ be a normal projective 
surface with only quotient singularities and $A\subset S$ an ample divisor. Let $\sG$ 
be a $A$-semistable, reflexive sheaf on $S$. Then there exists a (non-canonical) 
increasing filtration 
\begin{equation}\label{eq:filter-cover}
0= \sG_0 \subset  \sG_1 \subset \sG_2 \subset \ldots \subset 
 \sG_{t-1} \subset  \sG_t = \sG,
  \end{equation}
where every $\sQ_i:= \sG_i / \sG_{i-1}$ is a torsion-free $\bQ$-sheaf, $A$-stable, 
verifying the equality $\mu_A(\sQ_i)=\mu_A(\sG)$.

\end{prop}

\begin{proof} Let $p: \widehat S\to S$ be the global, Galois, finite morphism 
in Subsection~\ref{global}, constructed from the natural $\bQ$-structure associated to $S$, 
with $L:=\Gal(\widehat S/S)$. Denote $\widehat \sG:=p^{[*]}(\sG)$. 
Notice that the locally-free, $L$-sheaf $\widehat \sG$
is, with respect to $\widehat A:=p^*(A)$, semistable (see~\cite[Lem~3.2.2]{HL10}) but not $L$-stable, 
as the reflexive pull-back of the saturated, destablizing subsheaf of $\sG$  
destablizes $\widehat \sG$ as a $L$-subsheaf. 

\begin{subclaim}[First term of $\bQ$-$\rm{JH}_{\bullet}$]\label{claim:Q-JH} In the above setting, there exists a $L$-stable, 
saturated, $L$-subsheaf $\widehat \sG_1$ of $\widehat \sG$ with $\mu_{\widehat A}(\widehat \sG_1)=
\mu_{\widehat A}(\widehat \sG)$. 

\end{subclaim}
\noindent \emph{Proof of Claim~\ref{claim:Q-JH}.} Aiming for a contradiction, assume that 
there is no saturated, destabilizing $L$-subsheaf of $\widehat \sG$ that is $L$-stable.
Let $\widehat \sF$ be a saturated, destabilizing $L$-subsheaf. 
Notice that $\widehat \sF$ is, away from some finite number of isolated points, pull-back 
of a locally-free sheaf on $S$ (see~\cite[Prop.~2.16]{GKPT15}). 
More precisely the sheaf $\sF:=p_*(\widehat \sF)^L$ on $S$, formed 
by taking the $L$-invariant sections of $\widehat \sF$ , 
is reflexive and we have 
$\widehat \sF:=p^{[*]}(\sF)$. It follows that $\sF$ is semistable subsheaf of $\sG$ with 
$\mu_A(\sF)= \mu_A(\sG)$. By repeating this argument we can thus 
construct a decreasing sequence of chain of sheaves with strictly smaller
ranks and with equal slope which has to stablize (as a chain), contradicting the assumption 
that $\widehat \sG$ does not have a saturated, $L$-stable, destabilizing, $L$-subsheaf.
This proves Claim~\ref{claim:Q-JH}.

By using Claim~\ref{claim:Q-JH} repeatedly we now obtain an increasing 
filtration of $\widehat \sG$ 

  \begin{equation}\label{JH1}
0=\widehat \sG_0\subset  \widehat \sG_1\subset \widehat \sG_2 \subset \ldots \subset \widehat \sG_{t-1}\subset \widehat \sG_{t}=\widehat \sG,
      \end{equation}
where $\widehat \sQ_i:= \widehat \sG_i / \widehat \sG_{i-1}$ is a torsion-free, $L$-stable, 
$L$-sheaf with $\mu_{\widehat A}(\widehat \sQ_i)=\mu_{\widehat A}(\sG)$.
Noticing that $p_*(\cdot)^L$ is an exact functor, the required filtration~\ref{eq:filter-cover} can be 
constructed by taking the $L$-invariant sections of the filtration~\ref{JH1}.

\end{proof}

The next two lemma will allow us to detect when a filtration by torsion-free 
sheaves defines an extension by locally-free sheaves. These are 
technical tools that we shall need in the course of the proof of 
Theorem~\ref{finite-resolution}.

\begin{lem}[Reflexivity of quotients]\label{lem:reflex} Let $T$ be a normal projective surface with only 
quotient singularities,  
$H$ an ample divisor and 
$\sG$ a reflexive, $\sO_T$-module sheaf on $T$. Assume  
$\widehat \ch_1(\sG)\cdot H=0$ and $\widehat \ch_2(\sG)=0$. Let

\begin{equation}\label{filt:1}
0=\sG_0\subset   \sG_1\subset  \sG_2 \subset \ldots \subset 
\sG_{t-1}\subset  \sG_{t}= \sG
      \end{equation}
be an extension of $\sG$ by $H$-semistable, 
torsion-free, $\bQ$-sheaves $\sQ_i:=\sG_i/\sG_{i-1}$ 
with $\mu_H(\sQ_i)=0$. Then $\sQ_i$ is reflexive and that $\widehat \ch_2(\sQ_i)=0$, for all $i$.

\end{lem}

\begin{proof} The torsion-freeness of $\sQ_i$ implies that it differs from 
$\overline{\sQ_i}:=(\sQ_i)^{**}$ at most on a codimension two 
subset of $T$ so that $\overline{\sQ_i}$ inherits the stability conditions of $\sQ_i$.
By \cite[Lem.~10.9]{Meg92}, we have $\widehat \chern_2(\overline{\sQ_i})\leq \widehat \chern_2(\sQ_i)$ with equality if
and only if $\overline{\sQ_i}= \sQ_i$. As $\widehat \chern_1(\overline{\sQ_i})=\widehat \chern_1(\sQ_i)$, this is the same as $\widehat \ch_2(\sQ_i)\leq \widehat \ch_2(\overline{\sQ_i})$. Since $\widehat \chern_1(\overline{\sQ_i})\cdot H=0$, 
the Hodge-index theorem implies that 
\begin{equation}\label{Hodge-index}
\widehat \chern_1^2(\overline{\sQ_i})\leq 0.
\end{equation}
But as $\overline{\sQ_i}$ is $H$-semistable, we have, thanks to the Bogomolov inequality (\ref{Bogy}), that
\begin{align*} -2 \cdot \widehat\ch_2(\overline{\sQ_i})=
2\cdot \widehat \chern_2(\overline{\sQ_i})-\widehat \chern_1^2(\overline{\sQ_i}) &\geq \frac{-1}{r_i} \cdot \widehat \chern_1^2(\overline{\sQ_i})\ \ \ \ \ \ \ (r_i=\rank(\overline{\sQ_i}))\\
& \geq 0, \ \ \text{by the inequality (\ref{Hodge-index})}.
\end{align*}
That is, $\widehat \ch_2(\overline{\sQ_i})\leq 0$ for all $i$. Together, by the additivity of the second $\mathbb{Q}$-Chern character for direct sums, we have 
\begin{equation}\label{eq:2}
0=\widehat \ch_2(\sG)=\sum_i\widehat \ch_2(\sQ_i)\leq \sum_i\widehat \ch_2(\overline{\sQ_i})\leq 0,
\end{equation}
and thus $\widehat \ch_2(\overline{\sQ_i})=\widehat \ch_2(\sQ_i)=0$. It follows that $\overline{\sQ_i}=\sQ^i$, that is $\sQ_i$ is reflexive, for all $i$.

\end{proof}

\begin{lem}[Locally-freeness of semistable $\bQ$-sheaves]\label{lem:local-free} Let $T$ be 
a projective surface with only quotient singularities and 
$\widehat \pi_1(T_{\reg})\cong \widehat \pi_1(T)$. Let $H$ be an ample 
divisor. Any $H$-semistable reflexive sheaf $\sG$ with $\widehat \ch_1(\sE)\cdot H=0$
 and $\widehat \ch_2(\sG)=0$ is locally-free. 

\end{lem}

\begin{proof} If $\sG$ is $H$-stable, then the $\sG|_{T_{\reg}}$ is flat by 
the assertion~\ref{invariant-HE} 
and therefore is locally-free thanks to the assumption 
$\widehat \pi_1(T_{\reg})\cong \widehat \pi_1(T)$, cf.~\cite[Thm.~1.14]{GKP14a}. 
For the case where $\sE$ is strictly semistable, we consider a $\bQ$-Jordan-H\"older 
filtration of $\sE$:

\begin{equation}\label{filt:3}
0=\sG_0\subset   \sG_1\subset  \sG_2 \subset \ldots \subset 
\sG_{t-1}\subset  \sG_{t}= \sG.
      \end{equation}
Denote the torsion-free quotients $\sG_i/\sG_{i-1}$ by $\sQ_i$. As each $\sQ_i$ is 
stable with respect to $H$ and verifies the equality $\widehat \ch_1(\sQ_i)\cdot H=0$, Lemma~\ref{lem:reflex} applies and we 
find that $\sQ_i$ is reflexive and $\widehat \ch_2(\sQ_i)=0$, for all $i$, that is each $\sQ_i$ is an $H$-stable, reflexive 
sheaf with vanishing Chern classes. According to the assertion~\ref{invariant-HE}, it follows that 
each $\sQ_i|_{T_{\reg}}$ is flat. The assumption $\widehat \pi_1(T_{\reg})\cong \widehat \pi_1(T)$ now implies that every $\sQ_i$ is locally-free and therefore so is $\sG$.

\end{proof}

\

\subsection{Proof of Theorem~\ref{finite-resolution}}
We first consider the case when $X=S$ is a surface. Let $f:T \to S$ be the Galois, \'etale in codimension-$1$ cover given in Theorem~\cite[Thm.~1.4]{GKP14a}, with Galois group $G$, over which locally-free flat sheaves on $\widehat T_{\reg}$ extend (\cite[Thm.~1.14]{GKP14a}). Denote the reflexive  $G$-sheaf $f^{[*]}\sE$ by $\widehat \sE$, and notice that $\widehat \sE$ is semistable, cf.~\cite[Lem.~3.2.2]{HL10}.
 Let
\begin{equation}\label{JH}
0=\widehat \sE_0\subset  \widehat \sE_1\subset \widehat \sE_2 \subset \ldots \subset  \widehat \sE_{t-1}\subset \widehat \sE_{t}=\widehat \sE
      \end{equation}
be the $G$-equivariant, $\bQ$-Jordan-H\"older filtration of $\widehat \sE$ constructed in Proposition~\ref{Q-JH}, 
that is each quotient $\widehat \sQ_i:=\widehat \sE_i / \widehat \sE_{i-1}$ is a torsion-free, $G$-stable, $\bQ$-sheaf with the same slope, equal to $\mu(\widehat \sE)$. 
According to Lemma~\ref{lem:reflex}, we find that each $\widehat \sQ_i$ is reflexive and 
therefore, by Lemma~\ref{lem:local-free}, is locally-free. 
This in turn implies that $\widehat \sE$, as an extension locally-free sheaves, is 
also locally-free, as an analytic sheaf. 
Let $\pi: \wtilde T \to T$ be 
a $G$-equivariant resolution. 
The local freeness of the terms in the filtration of $\widehat \sE$ as well as in its grading and
the $G$-equivariance implies that the filtration lifts to a locally free filtration of $\pi^*(\widehat \sE)$ with terms $\pi^*(\widehat \sE_i)$.
By Proposition~\ref{lift}, the locally free sheaves $\pi^*(\widehat \sQ_i)$ are all $G$-stable with respect to a fixed $G$-invariant polarization. As before, every $G$-stable grading $\pi^*(\widehat \sQ_i)$ of the filtration of $\pi^*(\widehat \sE)$ induced by that of $\widehat \sE$ admits a $G$-invariant unitarily flat connection. Hence, $\pi^*(\widehat \sE)$ is an iterated extension by unitarily flat and locally free sheaves which is therefore, by Simpson's correspondence~\cite[Cor.~3.10]{Sim92} (and the remarks immediately after that corollary), endowed with a unique holomorphic flat connection. Since the filtration is $G$-equivariant, the uniqueness of  this flat connection implies that it is $G$-invariant. This yields a $G$-invariant flat connection on 
${\widehat \sE}|_{T_{\reg}}$ (see the arguments right before the assertion~\ref{invariant-HE}). 
As a result the restriction ${\sE}|_{S_{\reg}}$ 
comes from a representation of $\pi_1(S_{\reg})$.
This proves Theorem~\ref{finite-resolution} when $\dim(X)=2$. \\[-2mm]

For higher dimensional $X$, 
exactly the same argument as that of Section~\ref{KH} gives the semistable analog of Theorem~\ref{NS-result} and Corollary~\ref{SirSimon}. More precisely, given a reflexive sheaf, 
semistable with respect to $(H_1, \ldots, H_{n-1})$ and with zero Chern characters, let 
$S$ be the complete intersection surface defined in Setup~\ref{enum:properties}. According to 
Property~\ref{enum:properties}.\ref{prop:2}, the restriction $\sE|_S$ is reflexive and semistable. 
As $S$ is klt (Property~\ref{enum:properties}.\ref{prop:1}), the above arguments 
for the case $\dim(X)=2$ imply that $\sE|_{S_{\reg}}$ is flat. Property~\ref{enum:properties}.\ref{prop:3} 
now implies that there exists a reflexive sheaf $\sG\in \mathcal V_X$, where $\mathcal V_X$ 
is the family defined in Setup~\ref{setup}, such that $\sG|_S \cong \sE|_S$. According 
to Property~\ref{enum:properties}.\ref{prop:4} it thus follows that $\sG\cong \sE$, i.e. 
$\sE|_{X_{\reg}}$ is flat. Theorem~\ref{finite-resolution} now follows from~\cite[Thm.~1.14]{GKP14a}.

\subsection{A proof of Theorem~\ref{main} via semistability}

According to Proposition~\ref{SS-KLT} the tangent sheaf $\sT_X$ is generically semi-positive, that is, as $K_X\equiv 0$, $\sT_X$ is semistable independent of the choice of polarization. Therefore, Theorem~\ref{finite-resolution} says that there exits a finite quasi-\'etale cover $f:Y \to X$ such that $f^{[*]}\sT_X$ is locally-free. But since $f$ is unramified in codimension-$1$, we have the sheaf isomorphism $\det \sT_Y\cong \det(f^{[*]}\sT_X)$ by the ramification formula. As a result, the natural inclusion of reflexive sheaves $\sT_Y\to f^{[*]}\sT_X$ is an isomorphism and $\sT_Y$ is locally-free. 
The rest of the proof is identical to that of the polystable case in Section~\ref{polysection}. \qed\\[-2mm]

We now briefly explain the equivalence of the two sets of conditions $\{$(\ref{conditions:1}.\ref{first}), (\ref{conditions:1}.\ref{second})$\}$ and $\{$(\ref{conditions:2}.\ref{first0}), (\ref{conditions:2}.\ref{second0})$\}$. 

\begin{explanation}\label{explaining} Assume that $X$ is a klt projective variety verifying condition (\ref{conditions:1}.\ref{first}), i.e. $K_X\equiv 0$. Then by Lemma.~\ref{SS-KLT} we know that there exists a finite quasi-\'etale morphism $f:Y\to X$ such that $\sT_Y$ is generically semipositive. But $f$ being \'etale in codimension-$1$ implies that $\sT_Y=f^{[*]}\sT_X$ and thus $\sT_X$ is also generically semipositive by Lemma~\ref{gsp}, establishing condition~(\ref{conditions:2}.\ref{first0}). As $\widehat \chern_1(\sT_X)$ agrees with $\chern_1(\sT_X)=\chern_1(K_X)$ in codimension-$1$ by construction, we have $\widehat \chern_1(\sT_X)|_S\equiv 0$ for a general surface $S$ cut out by $n-2$ very ample divisors and thus $\widehat \chern_1^2(\sT_X)\cdot S= 0$. Hence (\ref{conditions:2}.\ref{second0}) follows from (\ref{conditions:1}.\ref{second}). Conversely, assume that the conditions (\ref{conditions:2}.\ref{first0}) and (\ref{conditions:2}.\ref{second0}) hold. Then, by Theorem~\ref{finite-resolution} we know that there exists a finite quasi-\'etale morphism $f:Y\to X$ such that 
$\sT_Y=f^*(\sT_X)$ is flat and locally free so that it has vanishing (orbifold) Chern classes and, in particular, $K_Y$ is numerically trivial.   The ramification formula $K_Y=f^*(K_X)$ (together with the projection formula) implies that $K_X\equiv 0$ and condition~(\ref{conditions:1}.\ref{second}) follows by Lemma~\ref{Q-projection}. \\[-3mm]

\end{explanation}

\
\

\providecommand{\bysame}{\leavevmode\hbox to2em{\hrulefill}\thinspace}
\providecommand{\MR}{\relax\ifhmode\unskip\space\fi MR }
\providecommand{\MRhref}[2]{%
  \href{http://www.ams.org/mathscinet-getitem?mr=#1}{#2}
}
\providecommand{\href}[2]{#2}


\begin{thebibliography}{GKKP11}

\bibitem[BCHM10]{BCHM10}
Caucher Birkar, Paolo Cascini, Christopher D.Hacon, and James McKernan, \emph{Existence of minimal models for varieties of log general type}, Journal of the AMS 23 (2010), 405-468.


\bibitem[BM95]{BM95}
E. Bierstone, P. Milman, \emph{Canonical desingularization in characteristic zero by blowing up the maximum strata of a local invariant}, Invent. Math. 128(2): 207--302.


\bibitem[Bog79]{Bog79}
F. Bogomolov, \emph{Holomorphic tensors and vector bundles on projective varieties}, 
Math. USSR Izv., 13 (1979), 499--555.


\bibitem[Don87]{Don87}
S. K. Donaldson, \emph{Infinite determinant, stable bundles and curvature.} Duke Math. J. \textbf{54} (1987), 231--247.

\bibitem[Fle84]{Fl84}
H. Flenner, 
\emph{Restrictions of semistable bundles on projective varieties.}
Comment. Math. Helv. 59 (1984), no. 4, 635--650.

%

\bibitem[GKKP11]{GKKP11}
D. Greb, S. Kebekus, S. Kov\'acs and T. Peternell, \emph{Differential forms on log canonical spaces.} Publications Math\'ematiques de l'IH\'ES, vol. 114, Number 1 (2011), 87--169.



\bibitem[GKP12]{GKP12}
D. Greb, S. Kebekus, T. Peternell, \emph{Singulr spaces with trivial canonical class.} To appear in Minimal models and extremal rays--proceedings of the conference in honour of Shigefumi Mori's 60th birthday, Advanced Studies in Pure Mathematics, Kinokuniya Publishing House, Tokyo.


\bibitem[GKP14a]{GKP14a}
D. Greb, S. Kebekus, T. Peternell, \emph{\'Etale fundamental group of Kawamata log terminal spaces, flat sheaves, and quotients of Abelain varieties.} Preprint arXiv:1307.5718. (To be published in Duke)



\bibitem[GKP14b]{GKP14b} D. Greb, S. Kebekus, S. Kov\'acs and T. Peternell, \emph{Movable Curves and Semistable Sheaves.} 
International Mathematics Research Notices, Oxford University Press (2015).


\bibitem[GKPT15]{GKPT15} D. Greb, S. Kebekus, T. Peternell and B. Taji, 
\emph{The Miyaoka-Yau inequality and uniformisation of canonical models.} 
Preprint: arXiv: arXiv:1511.08822. 



\bibitem[Gro65]{Gro65}
J. Dieudonn\'{e} and Grothendieck, Alexandre, 
\emph{\'{E}l\'{e}ments de g\'{e}om\'{e}trie alg\'{e}brique IV, \'{E}tude locale des sch\'{e}mas et des morphismes de sch\'{e}mas (Seconde Partie)},
Institut des Hautes \'{E}tudes Scientifiques publications math\'{e}matiques, 
Number 24, 1965. 


\bibitem[HL85]{HL85}
H. A, Hamn and Tr\'ang L\^e D\`ung, \emph{Lefschetz Theorems on quasi-projective varieties.} Bull. Soc. Math. France \textbf{113} (1985) 123--142.


\bibitem[Har80]{Har80}
R. Hartshorne, \emph{Stable reflexive sheaves}, Math. Ann., 254(2):121--176, 1980.




\bibitem[HL10]{HL10}
D. Huybrechts and M. Lehn, \emph{The geometry of moduli spaces of sheaves}, Cambridge Mathematical Library. Cambridge University, second edition, 2010.


\bibitem[Kaw92]{Kaw92}
Y. Kawamata,
\emph{Abundance theorem for minimal threefolds},
Invent. Math. 108 (1992), no. 2, 229 -- 246. 


\bibitem[Kle66]{Kle66}
Steven L. Kleiman, \emph{Toward a numerical theory of ampleness}, Ann. of Math. (2) 84 (1966), 293--344. 


\bibitem[Kob87]{Kob87}
S. Kobayashi, \emph{Differential geometry of complex vector bundles.} Publications of the Mathematical Society of Japan, vol. 15, Iwanami Shoten and Princeton University Press, Princeton, NJ, 1987, Kan\^o Memorial Lectures, 5.

\bibitem[Lan04]{Lan04}
A. Langer,  \emph{Semistable sheaves in positive characteristic}, Ann. of Math. (2) 159 (2004), no. 1, 251--276.

\bibitem[LP97]{LP97}
J. Le Potier, \emph{Lecture on Vector Bundles}, Cambridge Studies in Advanced Mathematics 52.

\bibitem[KM98]{KM98}
J. Koll\'ar, S. Mori, \emph{Birational geometry of algebraic varieties.} Cambridge Tracts in Mathematics, vol. 134, Cambridge University Press, Cambridge, 1998, With the collaborations of C.H Clemens and A. Corti, Translated from the 1998 Japanese original.



\bibitem[Meg92]{Meg92}
G. Megyesi, \emph{Chern classes of $\bQ$-Sheaves.} Flips and Abundance for Algebraic $3$-folds, Ch. 10, Ast\'ersque, vol. 211 (1993)
 

\bibitem[MR82]{MR82}
V. D. Mehta, A. Ramanathan, \emph{Semistable sheaves on projective varieties and their restriction to curves}, Math. Ann. \textbf{258} (1982) 213--224.

\bibitem[MR84]{MR84}
V. D. Mehta, A. Ramanathan, \emph{Restriction of stable sheaves and representations of the fundamental group},
Invent. Math. 77 (1984), no. 1, 163--172. 

\bibitem[Miy87]{Miy87}
Yoichi Miyaoka, \emph{The {C}hern classes and {K}odaira dimension of a minimal
  variety}, Algebraic geometry, Sendai, 1985, Adv. Stud. Pure Math., vol.~10,
  North-Holland, Amsterdam, 1987, pp.~449--476. \MR{89k:14022}

\bibitem[Miy85]{Miy85}
\bysame, \emph{Deformations of a morphism along a foliation and applications},
  Algebraic geometry, Bowdoin, 1985 (Brunswick, Maine, 1985), Proc. Sympos.
  Pure Math., vol.~46, Amer. Math. Soc., Providence, RI, 1987, pp.~245--268.
  \MR{MR927960 (89e:14011)}
  
  \bibitem[Mum83]{Mum83}
D. Mumford, \emph{Towards an enumerative geometry of the moduli space of curves.} Arithmetic and Geometry, papers dedicated ti I.R. Shafarevich, vol. II Geometry (M. Artin and J. Tate, eds), Progr. Math., vol. 36, Birkh\"auser, Boston, 1983, pages 271--326.


\bibitem[Nak04]{Nak04}
 N. Nakayama, \emph{Zariski decomposition and abundance.} MSJ Memoirs, 14, 
 Mathematical Society of Japan, Tokyo, 2004. 


\bibitem[NS65]{NS65}
M. S. Narasimhan, C. S. Seshadri, \emph{Stable and unitary vector bundles on compact Riemann surface.} Annales of Mathematics. Second Series 82: 540--567.


\bibitem[Pra93]{Pra93}
O. Garcia-Prada, \emph{Invariant connections and vortices.} Communications in Mathematical Physics 156 (1993), no. 3, 527--546.


\bibitem[Sat56]{Sat56}
I. Satake, \emph{On a generalization of the notion of manifold, Proceeding of the national academy of sciences}, vol. 42, pp 359--363 (1956).


\bibitem[SBW94]{SBW94}
N. I. Shepherd-Barron, P. M. H. Wilson, \emph{Singular threefolds and numerically trivial first and second Chern classes.} J. Algebraic Geom., 3(2):265--281, 1994.

\bibitem[Sim92]{Sim92}
C. T. Simpson, \emph{Higgs bundles and local systems}, Inst. Hautes \'Etudes Sci. Publ. Math. {\textbf 75}, 5--95. 

\bibitem[UY86]{UY86} 
K. Uhlenbeck, S.-T Yau, \emph{On the existence of Hermitian Yang-Mills connections on stable vector bundles.} Comm. Pure Appl. Math. \textbf{39} (1986), 257--293; \emph{A note on our previous paper: On the existence of Hermitian Yang-Mills connections on stable vector bundles.} Comm. Pure Appl. Math. \textbf{42} (1989), 703 --707.

\bibitem[Xu12]{Xu12}
C. Xu, \emph{Finiteness of algebraic fundamental groups.} 
Compos. Math., \textbf{150}, no. 3, 409--414. 

\bibitem[Yau78]{Yau78}
S.-T. Yau, \emph{On the Ricci curvature of compact K\"ahler manifold and the complex Monge-Amp\`ere equation.} I. Comm. Pure Appl. Math., 31(3):339--411, 1978.


\end{thebibliography}
\end{document}